\documentclass[11pt, reqno]{amsart}
\usepackage[top=3.75cm, bottom=3cm, left=3cm, right=3cm]{geometry}
\usepackage{url}
\usepackage{textcomp}
\usepackage{amsmath}
\usepackage{amsthm}
\usepackage{amssymb}
\usepackage{stmaryrd}
\usepackage[french, british]{babel}
\usepackage[utf8]{inputenc} 
\usepackage[T1]{fontenc}
\usepackage{tikz}
\usepackage{tikz-cd}     
\usepackage{centernot}
\usepackage{mathrsfs}
\usepackage[all]{xy}
\usepackage{calligra}
\usepackage{nicefrac}
\usepackage{mathtools}
\usepackage{hyperref}   
\usepackage[capitalise]{cleveref}	
\hypersetup{
	colorlinks=true,
	linkcolor=blue,
	filecolor=magenta,      
	urlcolor=cyan,
}
\usepackage{todonotes}    
\usepackage{enumitem}
\usepackage{graphicx}
\graphicspath{ {images/} }

\usepackage{multirow}
\usepackage{enumitem}

\theoremstyle{plain}
\newtheorem{theorem}{Theorem}[section]
\newtheorem*{theorem*}{Theorem}
\newtheorem*{theoreme*}{Théorème}

\newtheorem*{question*}{Question}

\newtheorem{lemma}[theorem]{Lemma}
\newtheorem{corollary}[theorem]{Corollary}
\newtheorem{proposition}[theorem]{Proposition}

\theoremstyle{definition}
\newtheorem{remark}[theorem]{Remark}

\theoremstyle{definition}
\newtheorem{definition}[theorem]{Definition}

\DeclareMathOperator{\Aut}{Aut}

\DeclareMathOperator{\Span}{Span}
\DeclareMathOperator{\rk}{rk}
\DeclareMathOperator{\Gr}{Gr}
\DeclareMathOperator{\GL}{GL}

\DeclareMathOperator{\SL}{SL}

\DeclareMathOperator{\Hom}{Hom}
\DeclareMathOperator{\coker}{coker}
\DeclareMathOperator{\im}{Im}
\DeclareMathOperator{\Id}{Id}

\DeclareMathOperator{\Spec}{Spec}

\DeclareMathOperator{\Supp}{Supp}
\DeclareMathOperator{\Pic}{Pic}

\DeclareMathOperator{\Br}{Br}
\DeclareMathOperator{\Quot}{Quot}

\DeclareMathOperator{\SheafHom}{\mathcal{H\kern -1pt}\textit{om}} 

\DeclareMathOperator{\NS}{NS}
\DeclareMathOperator{\Nef}{Nef}
\DeclareMathOperator{\Mov}{Mov}
\DeclareMathOperator{\Pos}{Pos}
\DeclareMathOperator{\Amp}{Amp}

\DeclareMathOperator{\Ext}{Ext}
\DeclareMathOperator{\Tor}{Tor}
\DeclareMathOperator{\ext}{ext}

\DeclareMathOperator{\ch}{ch}
\DeclareMathOperator{\td}{td}

\DeclareMathOperator{\D}{D}

\DeclareMathOperator{\Stab}{Stab}

\def\dar[#1]{\ar@<2pt>[#1]\ar@<-2pt>[#1]}

\DeclareMathOperator{\Coh}{\mathbf{Coh}}

\newcommand\matC{\mathbb{C}}

\newcommand\matH{\mathbb{H}}

\newcommand\matP{\mathbb{P}}

\newcommand\matR{\mathbb{R}}

\newcommand\matZ{\mathbb{Z}}

\newcommand\calA{\mathcal{A}}
\newcommand\calB{\mathcal{B}}
\newcommand\calC{\mathcal{C}}

\newcommand\calE{\mathcal{E}}
\newcommand\calF{\mathcal{F}}
\newcommand\calG{\mathcal{G}}
\newcommand\calH{\mathcal{H}}
\newcommand\calI{\mathcal{I}}

\newcommand\calM{\mathcal{M}}

\newcommand\calO{\mathcal{O}}
\newcommand\calP{\mathcal{P}}

\newcommand\calR{\mathcal{R}}
\newcommand\calS{\mathcal{S}}

\newcommand\calU{\mathcal{U}}

\newcommand\calW{\mathcal{W}}

\newcommand\calY{\mathcal{Y}}

\newcommand{\fonction}[5]{\begin{array}{lrcl} 
		#1: & #2 & \longrightarrow & #3 \\
		& #4 & \longmapsto & #5 \end{array}}

\newcommand{\mysetminusD}{\hbox{\tikz{\draw[line width=0.6pt,line cap=round] (3pt,0) -- (0,6pt);}}}
\newcommand{\mysetminusT}{\mysetminusD}
\newcommand{\mysetminusS}{\hbox{\tikz{\draw[line width=0.45pt,line cap=round] (2pt,0) -- (0,4pt);}}}
\newcommand{\mysetminusSS}{\hbox{\tikz{\draw[line width=0.4pt,line cap=round] (1.5pt,0) -- (0,3pt);}}}
\newcommand{\mysetminus}{\mathbin{\mathchoice{\mysetminusD}{\mysetminusT}{\mysetminusS}{\mysetminusSS}}}

\title[Moduli of sheaves on Fano and K3 of genus $9$]{Moduli spaces of sheaves on Fano threefolds and K3 surfaces of genus $9$}

\date{}
\author{Dominique Mattei}

\address{Universit\"at Bonn, Endenicher Allee 60, 53115 Bonn, Germany.}
\email{dmattei@math.uni-bonn.de}

\begin{document}
		\maketitle
		
		\begin{abstract} 
			A complex smooth prime Fano threefold $X$ of genus $9$ is related via projective duality to a quartic plane curve $\Gamma$. We use this setup to study the restriction of rank $2$ stable sheaves with prescribed Chern classes on $X$ to an anticanonical $K3$ surface $S\subset X$. Varying the threefold $X$ containing $S$ gives a rational Lagrangian fibration
			$\calM_S[2,1,3] \dashrightarrow \matP^3$
			with generic fibre birational to the moduli space $\calM_X(2,1,7)$ of sheaves on $X$.
			Moreover, we prove that this rational fibration extends to an actual fibration on a birational model $\calM$ of $\calM_S[2,1,3]$.
			
			In a last part, we use Bridgeland stability conditions to exhibit all $K$-trivial smooth birational models of $\calM_S[2,1,3]$, which consist in itself and $\calM$. We prove that these models are related by a flop, and we describe the positive, movable and nef cones of $\calM_S[2,1,3]$.
		\end{abstract}

		\section{Introduction}
		
		Let $X$ be a smooth complex projective Fano threefold of index $1$ and Picard group generated by an ample divisor $H_X$. The moduli spaces $\calM_X(2,1,c_2,c_3)$ of semistable sheaves $F$ with rank $2$ and Chern classes $c_1(F)=c_1(H_X)$, $c_2(F)=c_2$, $c_3(F)=c_3$ attracts a lot of attention since the pioneer work of Barth for $X=\matP^3$ \cite{BarthPropStableRank2}. Using homological methods, Brambilla and Faenzi \cite{BrambillaFaenzig7} construct and describe a generically smooth and irreducible component $M_X(d)\subset \calM_S(2,1,d)\coloneqq \calM_S(2,1,d,0)$, whose general element $F$ is locally free and satisfies $\Ext^2(F,F)=0$.
		
		The very general hyperplane section $S\in|H_X|$ is a K3 surface with Picard group generated by the restriction $H_S\coloneqq H_X|_S$. Based on a result of Tyurin (see \cite{BeauvilleFanoThreefoldK3}), the authors  in \cite{BrambillaFaenzig7} show that there exists some open $M_X(d)^o\subset M_X(d)$ for which the restriction
		$$\mathsf{res}\colon M_X(d)^o \to \calM_S(2,1,d), \ F\mapsto F_S$$
		is an immersion (i.e. a morphism with injective differential), and its image is a (possibly singular) Lagrangian subvariety with respect to the symplectic structure on $\calM_S(2,1,d)$.
		
		For the specific case of genus $g(X)\coloneqq H_X^3/2+1 =9$, the Fano $X$ is related to a quartic plane curve $\Gamma$ by \textit{Homological Projective Duality} (see \S \ref{sectionHPD}), in particular there is an an embedding
		$\phi\colon\D^b(\Gamma) \hookrightarrow \D^b(X)$ between the derived categories of $\Gamma$ and $X$.
		In \cite{BrambillaFaenzig9}, the authors use the right adjoint $\phi^!\colon \D^b(X) \to \D^b(\Gamma)$ of this functor to prove that the map 
		$$\calM_X(2,1,7) \to \Pic^2(\Gamma), \ F\mapsto \phi^!F$$
		gives an isomorphism of $\calM_X(2,1,7)$ with the blow-up of $\Pic^2(\Gamma)$ along a curve isomorphic to the Hilbert scheme $\calH_1^0(X)$ of lines contained in $X$; the exceptional divisor consists of the sheaves of $\calM_X(2,1,7)$ which are not globally generated. We use this construction to prove our first main theorem. For short, let us denote $\calM_X\coloneqq \calM_X(2,1,7)$ and $\calM_S\coloneqq \calM_S(2,1,7)$.
		
		\begin{theorem}[= Theorem~\ref{thmresinj} and Theorem~\ref{thmrestotal}]
			The restriction $\mathsf{res}\colon\calM_X \to \calM_S$ is injective on the set of globally generated sheaves (in particular, it is generically injective). The image $\mathsf{res}(\calM_X)$ is a Lagrangian subvariety of $\calM_S$ with finitely many singular points, each of which have exactly $2$ preimages in $\calM_X$.
		\end{theorem}
		
		In \cite{BeauvilleFanoThreefoldK3}, Beauville asks (in a more general context) if the subvariety $\mathsf{res}(\calM_X)$ moves in a family on $\calM_S$, and if this family describes a (rational) Lagrangian fibration $\calM_S \to B$. This is the content of \S \ref{sectionLagrangianFibration}.
		
		\begin{theorem}[= Corollary~\ref{corratlagfib}]\label{ThmIntroRatLagFib}
			The Fano threefold $X$ moves in a $3$-dimensional open family $\frak{X}\to \calW\subset \matP^3$ parametrizing smooth Fanos containing $S$. There exists a rational Lagrangian fibration $\calM_S\dashrightarrow \matP^3$ such that the fibre over a general point $[X_t]\in\matP^3$ is the open of globally generated sheaves in $\calM_{X_t}$.
		\end{theorem}
		
		In fact, we prove that the relative moduli space $\calM_{\mathfrak{X}/\calW}$ of the family $\mathfrak{X}\to \calW$ is birational to $\calM_S$. The map $\calM_X\dashrightarrow \matP^3$ is given by sending a globally generated sheaf $F\in\calM_S$, which is the restriction of some sheaf in $\calM_{X_t}$, to $[X_t]\in\matP^3$.
		
		Homological projective duality associates to the K3 surface $S$ another polarized K3 surface $(S',H')$, equipped with a sheaf of Azumaya algebras $\calA$ (or equivalently with a Brauer class $\alpha\in\Br(S')$). This permits to define a moduli space of semistable $\alpha$-\textit{twisted} torsion sheaves $\calM_{(S',\alpha)}\coloneqq \calM_{(S',\alpha)}(0,H',2)$, equipped with a structure of Lagrangian fibration
		$$\calM_{(S',\alpha)} \to \matP^3=|H'|, \ L \mapsto \Supp(L),$$
		such that the fibre over a smooth curve $[\Gamma]\in |H'|$ is isomorphic to $\Pic^2(\Gamma)$. In fact, the smooth curves in $|H'|$ are precisely the plane quartics HP-dual to the Fanos in $\mathfrak{X}$.
		The following result can be thought as a \textit{compactification} of the rational fibration on $\calM_S$.
		
		\begin{theorem}[= Theorem~\ref{thmactuallagfibrationbiratmodel}]
			The linear system $|H'|$ naturally identifies with the projective space $\matP^3$ of Theorem~\ref{ThmIntroRatLagFib}, and there exists a birational map $\calM_S \dashrightarrow \calM_{(S',\alpha)}$ over $\matP^3$.
		\end{theorem}
		
		It is known that the \textit{intermediate Jacobian} of $X$ identifies with the Jacobian $\Pic^0(\Gamma)$, for $\Gamma$ the HP-dual curve of $X$. In particular, the fibration $\calM_{(S',\alpha)}\to\matP^3$ can be thought as a compactification of the \textit{twisted} intermediate Jacobian fibration associated the family $\mathfrak{X}\to \calW$.
		
		Eventually, we use the remarkable results of Bayer and Macr\`i \cite{BMMMPwallcrossing} to study the birational models of $\calM_S$. 
		
		\begin{theorem}[= Theorem~\ref{ThmKtrivialBirModels}]
			The moduli spaces $\calM_S$ and $\calM_{(S',\alpha)}$ are the only $2$ birational models of $\calM_S$. They are not isomorphic and are related by a flop along a $\matP^2$-bundle over $S$.
		\end{theorem}
		
		In addition, we obtain a complete description of the positive, movable and nef cones of $\calM_S$, see \S \ref{SectionMovNefMS}.

		\subsection*{Notations and conventions}
		
		Throughout the paper, we work over $\matC$. By variety, we mean an integral separated scheme of finite type. By the word sheaf, resp. vector bundle, we mean a coherent sheaf, resp. finite rank locally free sheaf.
		We denote $i_{RT}\colon  R \hookrightarrow T$ a closed immersion between two schemes $R$ and $T$, and $F_R$ the restriction to $R$ of a sheaf $F$ on $T$. We denote $\D^b(R)\coloneqq \D^b(\Coh(R))$ the bounded derived category of coherent sheaves on $R$.
		The dual of a sheaf $F$ is denoted $F^*$, and the derived dual is denoted $F^\vee$. Recall that both coincide when $F$ is locally free.
		For any $i\in\matZ$ and sheaves $F,G$, we denote $\ext^i(F,G)\coloneqq \dim \Ext^i(F,G)$.

		\subsection*{Acknowledgements}The results of this article are part of my PhD project. I would like to thank my advisor Marcello Bernardara for his invaluable guidance, encouragement, and support throughout this work. I am grateful to Kota Yoshioka for his helpful feedback and for pointing out a mistake in the the first version of the paper. I deeply thank Daniele Faenzi for answering my numerous questions and for valuable conversations, and Arend Bayer, Daniel Huybrechts and Emanuele Macr\`i for reading and commenting previous version of this paper. I also thank \'Angel David R\'ios Ortiz for interesting questions and remarks. I am supported by the ERC Synergy Grant HyperK,	agreement ID 854361.

		\section{Preliminaries}\label{sectionsetup}

		\subsection{The Fano of genus $9$}\label{setupFanog9}
		
		We start with a prime Fano threefold of genus $9$ constructed as follows (see \cite{IlievRanestadGeometryOfLG36} for more details).
		First, we consider $\Sigma=LG(3,6)$, the Lagrangian Grassmannian of $3$-dimensional subspaces of a $6$-dimensional vector space $V$ which are isotropic with respect to a skew-symmetric $2$-form $\omega$. The manifold $\Sigma$ embeds in $\matP^{13} = \matP V_{14}$, with $V_{14}=\ker(\Lambda^2V\xrightarrow{\omega} \matC)$. Now we define $X$ as a $3$-codimensional linear section of $\Sigma$, that is
		$$X\coloneqq \Sigma\cap \matP V_{11}$$
		for $V_{11}\subset V_{14}$ an $11$-dimensional subvector space.
		The Fano $X$ has Picard group $\Pic(X)=\langle H_X \rangle$ with $H_X$ a hyperplane section in $\matP V_{11}$, and $-K_X=H_X$.
		A very general hyperplane section $S$ of $X$ is a smooth $K3$ surface of genus $9$ polarized by the restriction $H_S$ of $H_X$ to $S$, with $\Pic(S)=\langle H_S \rangle$.
		
		The manifold $\Sigma$ is equipped with a tautological homogeneous rank $3$ subbundle $\calU$. As we will principally study $X$, we denote $\calU$ again its restriction to $X$, and $\calU_S$ its restriction to $S$. By \cite[Lem. 3.1]{BrambillaFaenzig9}, both $\calU$ and $\calU_S$ are $\mu$-stable and  $c_1(\calU)=-H_X$.

		\subsection{Cohomology and moduli spaces of sheaves}\label{SectionCohomModuliSheaves}
		
		The integral cohomology groups $H^{k,k}(X)_\matZ$ of $X$ are generated by the hyperplane class $H_X$ ($k=1$), the class of a line $L_X\subset X$ ($k=2$) and the class of a closed point $P_X\in X$ ($k=3$). For now on, we will denote
		$$\calM_X(r,c_1,c_2,c_3)$$
		the coarse moduli space of Gieseker semistable sheaves on $X$ with rank $r$ and Chern classes $c_i$, $i=1,\dots,3$. If $c_3$ vanishes, we omit it from the notation. Moreover, we will often write integers instead of Chern class as they all are integral multiples of the corresponding generator of $H^{k,k}(X)_\matZ$.

		On $K3$ surfaces, it is more conveniant to use \textit{Mukai vectors} (we refer to \cite[Ch. 10]{HuybrechtsLecturesK3} for the general theory). That is, we denote $\calM_S[r,c,s]$ the moduli space of semistable sheaves on $S$ with rank $r$ and Chern classes $c_1(F)=cH_S$, $c_2(F)=\frac{c_1(F)^2}{2}-s+r$ (we use the bracket notation to avoid confusion).
		Recall that when $[r,c,s]$ is a primitive Mukai vector, the moduli space $\calM_S[r,c,s]$ is a \textit{hyperk\"ahler manifold} deformation equivalent to a Hilbert scheme of point on a $K3$ surface if not empty \cite{YoshiokaStabFMTransf}.
		
		In the present paper, we will focus our attention on $\calM_X(2,1,7)$ and $\calM_S[2,1,3]$.

		\vspace{1\baselineskip}

		\subsection{Technical lemmas}
		
		We gather here some useful lemmas.

		\begin{lemma}\label{lemmahigherpullbacktorsionfree}
			Let $X$ be a smooth integral projective variety, $S\subset X$ a smooth integral hypersurface. Let $F\in\Coh(X)$ be a coherent pure sheaf with $\dim \left(\Supp(F)\cap S\right) < \dim F$. Let $i\colon S\hookrightarrow X$ be the closed immersion. Then $L^ki^*F=0=\Tor_k(F,\calO_S)$ for all $k>0$.
		\end{lemma}
		
		\begin{proof}
			Note that $L^ki^*F=0$ for all $k>0$ if and only if $Li^*F$ is a sheaf. Since $i$ is a closed immersion, $Ri_*=i_*$ do not need to be derived, therefore $Li^*F$ is a sheaf if and only if $i_*Li^*F$ is a sheaf. By the projection formula, we have $i_*Li^*F\simeq F\otimes^L\calO_S$. Tensoring the exact sequence
			\begin{eqnarray}\label{EqBasicDivisorExSeq}
				0 \to \calO_X(-S) \to \calO_X \to \calO_S \to 0
			\end{eqnarray}
			by $F$, to prove both statements of the lemma we are reduced to show that $m\colon  F(-S) \to F$ is injective. Recall that a sheaf is pure if and only if all its associated points have the same dimension, in particular we can work locally and assume that $\calO_X(-S)$ is generated by a global function $f$ vanishing on $S$. Hence the kernel of
			$$F \xrightarrow{\times f} F$$
			is a subsheaf whose support $Z$ is contained in $\Supp(F)\cap S$, which have dimension smaller than the dimension of $\Supp(F)$ by assumptions. By purity of $F$, $Z$ must be empty, so $m$ is injective.
		\end{proof}
		
		\begin{remark}\label{RmkComputationsExtgroupsTools}
			Lemma~\ref{lemmahigherpullbacktorsionfree} in very useful when it comes to compute $\Ext$-groups between sheaves. Indeed, pick two sheaves $F,G$ as in the statement. The tensor product $-\otimes F$ applied to (\ref{EqBasicDivisorExSeq}) gives an exact sequence
			\begin{eqnarray}\label{EqRestrictionTemplate}
				0\to F(-S) \to F \to F_S \to 0
			\end{eqnarray} 
			and moreover $\Ext^i_X(G,F_S)=\Ext^i_S(G_S,F_S)$ by (derived) adjunction. The long exact sequence obtained by applying $R\Hom_X(G,-)$ to (\ref{EqRestrictionTemplate}) permits to relate the Ext groups between $G$ and $F$ and the ones between $G_S$ and $F_S$.
		\end{remark}

		\begin{theorem}[Hoppe's criterion, \cite{bibhoppecriterion}]\label{hoppecriterion}
			Let $X$ be smooth projective variety over $\matC$ with Picard group generated by an ample line bundle $\calO_X(1)$. For any vector bundle $E$ of rank $r$ on $X$, we have
			\begin{itemize}
				\item If $H^0(X,(\Lambda^qE)_{\textnormal{norm}}(-1))=0$ for $0<q<r$, then $E$ is $\mu$-semistable.
				\item If $H^0(X,(\Lambda^qE)_{\textnormal{norm}})=0$ for $0<q<r$, then $E$ is $\mu$-stable, and the converse is true when $r=2$.
			\end{itemize}
			Here $E_{\textnormal{norm}}\coloneqq E(-k_E)$ with $k_E\in\matZ$ unique so that $-r+1 \leq c_1(E_{\textnormal{norm}})\leq 0$.
		\end{theorem}

		\subsection{Homological Projective Duality}\label{sectionHPD}

		In this section we will only consider the application of \textit{Homological Projective Duality} (HPD) in our case. For the general theory, we refer to \cite{KuznetsovHPD}.

		In \cite{KuznetsovHyperplaneSections}, the author proves (so-called \textit{incomplete}) HP duality for $\Sigma=LG(3,6)$. Namely, we consider the maps 
		$$f\colon \Sigma \hookrightarrow \matP V_{14} = \matP^{13} \text{ and } Y\coloneqq \Sigma^\vee \smallsetminus \mathbf{Z}\hookrightarrow \matP V_{14}^\vee ,$$
		where $\Sigma^\vee$ is the (classical) projective dual variety of $\Sigma$, which is a quartic hypersurface singular along a subvariety $\mathbf{Z}\subset \Sigma^\vee$ of codimension $3$.
		
		We obtain the following semiorthogonal decompositions, denoting $\Sigma_j\coloneqq \Sigma\cap L$, resp. $Y_j = Y \cap L^\perp$, for an admissible linear subspace $L\subset V$ (i.e.  that satisfies $\dim \Sigma\cap \matP L=\dim \Sigma-\operatorname{codim} L$) of dimension $j$:
		
		\begin{itemize}
			\item $\D^b(\Sigma) = \langle \calO_\Sigma(1), \calU_\Sigma^*(1), \calO_\Sigma(2), \calU_\Sigma^*(2), \calO_\Sigma(3), \calU_\Sigma^*(3), \calO_\Sigma(4), \calU_\Sigma^*(4)  \rangle$
			\item $\D^b(\Sigma_{11})=\langle \D^b(Y_{11}), \calO_{\Sigma_{11}}(1), \calU^*_{\Sigma_{11}}(1) \rangle$
			\item $\D^b(\Sigma_{10}) = \D^b(Y_{10},\calA_Y)$.
		\end{itemize}
		
		In these cases, $\Sigma_{11}$ is the Fano threefold $X$ as in \S \ref{setupFanog9}, $Y_{11}$ is a plane quartic, and $\Sigma_{10}$ and $Y_{10}$ are K3 surfaces of degree $16$ and $4$ respectively. Beware that $\D^b(Y_{10},\calA_Y)$ is the derived category of $\calA_Y$-module with respect to a sheaf of Azumaya algebra $\calA_Y$ over $Y$. Equivalently, we can use the equivalence $\D^b(Y_{10},\calA_Y) \simeq \D^b(Y_{10},\alpha)$ with the derived category of coherent sheaves \textit{twisted} by a Brauer class $\alpha\in\Br(Y_{10})$ provided by \cite{Caldararuthesis}.

		\begin{remark}\label{rmkfamilyHPD}
			Later in the text we will need the following observation. One could ask how the functor $\phi_r\colon \D^b(Y_r,\calA_Y) \to \D^b(\Sigma_r)$ (with $r$ for which it makes sense) varies when changing the linear section $L$. In fact, in \cite{KuznetsovHyperplaneSections} the author proves similar semiorthogonal decompositions when replacing $\Sigma_{r}$ with $\calS_{r}$ and $Y_{r}$ with $\calY_{r}$, where $\calS_{r}\subset \Sigma\times \Gr(r,V^\vee)$ and $\calY_r \subset Y\times \Gr(r,V^\vee)$ are the universal families of linear sections.
			
			In particular, the kernels $\widetilde{\calE}_r$ of the functors between $\D^b(\calS_r)$ and $\D^b(\calY_r,\calA_Y)$ are the pullback of an object $\calE\in\D^b(Q(X,Y))$ through the maps $\calS_r\times_{\Gr(r,V^\vee)} \calY_r \to Q(\Sigma,Y)$, where $Q(\Sigma,Y)$ is the incidence quadric of couples $(s,y)\in \Sigma\times Y$ with $f(s)\in g(y)$. Finally, the kernels $\calE_r$ obtained by base change $\Spec \matC \to \Gr(r,V^\vee)$, which correspond to the choice of a linear section, give the semiorthogonal decompositions defined above.
		\end{remark}
		
		%

		\subsection*{Notations for the HPD}\label{sectionnotationXGammaHPD}

		We introduce some notation for the next parts. 
		\begin{itemize}

			\item We denote $X= \Sigma_{11}$ the Fano threefold, $\Gamma\coloneqq Y_{11}$ the plane quartic curve, $S\coloneqq\Sigma_{10}$ and $S'\coloneqq Y_{10}$ the K3 surfaces. Note that $S$, resp. $\Gamma$, is a hyperplane section of $X$, resp. $S'$.
			
			\item We denote $\calE$ the object in $\D^b(Q(\Sigma,Y))$ which gives the HP-duality and by $\calE_{11}$, resp. $\calE_{10}$ its restriction to $X\times \Gamma$, resp. $S\times S'$.
			
			\item We denote by
			\begin{eqnarray*}
				\phi_{11} &:& \D^b(\Gamma) \hookrightarrow \D^b(X) \\
				\phi_{10} &:& \D^b(S',\alpha) \xrightarrow{\sim} \D^b(S).
			\end{eqnarray*}
			the Fourier-Mukai functors with kernel $\calE_{11}$ and $\calE_{10}$ respectively obtained by HP-duality. Note that $\phi_{11}$ is fully faithful and $\phi_{10}$ is an equivalence.
		\end{itemize}

		We need the following lemma which relates the different "paths" between the derived categories $\D^b(X)$ and $\D^b(S)$, as it reads on diagram
		\begin{center}
			\begin{tikzcd}
				\D^b(\Gamma) \ar[r,hook,"\phi_{11}"] \ar[d,"Ri_{\Gamma S',*}"'] & \D^b(X) \ar[d,"Li_{SX}^*"] \\
				\D^b(S',\alpha) \ar[r,"\sim","\phi_{10}"'] & \D^b(S)
			\end{tikzcd}
		\end{center}
		
		\begin{lemma}\label{lemisoway}
			We have an isomorphism of functors
			$$Li_{SX}^{*}\circ \phi_{11} \simeq \phi_{10} \circ (Ri_{\Gamma S'})_*$$
			from $\D^b(\Gamma)$ to $\D^b(S)$.
		\end{lemma}
		
		\begin{proof}

			It is a consequence of the adaptation of \cite[Ex. 5.12]{HuybrechtsFMTransform} to the case of twisted sheaves. The key point is that all involved functors are Fourier-Mukai (see \cite{CanonacoStellariTwistedFMFunctors}).
			
			In our case, we get $Li_{SX}^{*}\circ \phi_{11}\simeq \phi_\calA$ with $\calA\simeq  L(\Id_\Gamma \times Li_{SX})^*\calE_{11}$ and $\phi_{10} \circ (Ri_{\Gamma S'})_* \simeq \phi_\calB$ with $\calB \simeq L(i_{\Gamma S'}\times \Id_S)^*(\calE_{10})$. But by definition of $\calE_{10}$ and $\calE_{11}$, $\calA$ and $\calB$ are both isomorphic to $Lj^*\calE$ with $j\colon \Gamma \times S \to Q(Y,\Sigma)$.
		\end{proof}

		\section{Studying $\calM_X(2,1,7)$ and its restriction}\label{sectionStudyMxandrestriction}

		Fix $X=\Sigma\cap \matP V_{11}$ as defined in \S \ref{sectionsetup}, and fix $S$ a general hyperplane section. In particular, we assume that $S$ is a $K3$ surface with $\Pic(S)=\matZ \cdot H_S$ of genus $9$ and $S$ does not contain a line.
		We denote $\calM_X\coloneqq\calM_X(2,1,7)$, and $\calM_S\coloneqq\calM_S[2,1,3]$. The present paper is based on the following results from Brambilla and Faenzi.
		
		\begin{theorem}[\cite{BrambillaFaenzig9}, Thm. 5.1]\label{ThmBFg9MXisBlowUpPic2}
			The functor $\phi_{11}$ admits a right adjoint $\phi^!_{11}$. The latter induces a morphism
			$$\fonction{\varphi}{\calM_X(2,1,7)}{\Pic^2(\Gamma)}{F}{\phi^!_{11}F}$$
			which is a blow-up of $\Pic^2(\Gamma)$ along a subvariety isomorphic to the Hilbert scheme of lines in $X$. The exceptional divisor of $\varphi$ consists of the sheaves in $\calM_X(2,1,7)$ which are not globally generated.
		\end{theorem}
		
		\begin{proposition}[\cite{BrambillaFaenzig9}, Lem. $5.2$ and \cite{BrambillaFaenzig7}, Prop. 3.4 and 3.6]\label{PropDescriptionSheafMx217}
			Let $F\in\calM_X(2,1,7)$ be a sheaf. Then we have
			\begin{eqnarray*}
				H^k(X,F)=0 &\text{ for }& k=1,2,\\
				H^k(X,F(-1))=0 &\text{ for }& k=0,1,2,3 \\
				H^1(X,F(-t))=0 &\text{ for }& t\geq 1
			\end{eqnarray*}
			Moreover, either $F$ is locally free or $F^{**}\in\calM_X(2,1,6)$ is a stable \textit{vector bundle}, and there is a line $M_F\subset X$ and an exact sequence
			\begin{eqnarray}\label{EqFnotlocfreeFF**OM}
				0 \to F \to F^{**} \to \calO_{M_F} \to 0.
			\end{eqnarray}
			
			Furthermore, the following statements are equivalent:
			\begin{enumerate}
				\item the sheaf $F$ is not globally generated,
				\item the vector space $\Hom(\calU^\vee,F)$ is non-zero,
				\item\label{EqnEvalutationNonGloballyGenMx217} Denote $I\hookrightarrow F$ the image of the natural evalutation map $I=\im(ev\colon H^0(X,F)\otimes \calO_X \to F)$. Then $I\in\calM_X(2,1,8,2)$ and there is a line $L_F\subset X$ such that we have
				\begin{eqnarray}\label{EqFnotggIFOL}
					0 \to I \to F \to \calO_{L_F}(-1) \to 0. 
				\end{eqnarray}
				Moreover the sheaf $I$ admits a locally free resolution
				\begin{eqnarray}\label{EqFnotggOUI}
					0 \to \calO_X \to \calU^\vee \to I \to 0. 
				\end{eqnarray}
			\end{enumerate}
		\end{proposition}
		
		We prove that the pullback by $i_{SX}\colon S \hookrightarrow X$ gives a restriction morphism
		\begin{eqnarray}\label{EqnRESmorphism}
			\mathsf{res}\colon \calM_X \to \calM_S.
		\end{eqnarray}

		\begin{lemma}\label{LemmaStableSheavesAreMuStableAndRestriction}
			Let $F\in\calM_X$ be a sheaf. Then $F$ is $\mu$-stable, and its restriction $F_S$ to $S$ is also $\mu$-stable.
		\end{lemma}
		
		\begin{proof}
			We know that $F$ is (Gieseker)-semistable. Let $G\subset F$ be a subsheaf of rank $1$ and with first chern class $c_1(G)=aH$ such that $\mu(F)=\mu(G)$. Then $aH^3 = H^3/2$ which is impossible for $a\in\matZ$. Hence $F$ is $\mu$-stable.
			
			Consider the exact (by Lemma \ref{lemmahigherpullbacktorsionfree}) sequence
			$$0 \to F(-2) \to F(-1) \to F_S(-1) \to 0.$$
			If $F$ is locally free, Hoppe's criterion (\ref{hoppecriterion}) gives $H^0(X,F(-1))=0$, and $H^1(X,F(-2))=0$ by Proposition~\ref{PropDescriptionSheafMx217}. Hence $H^0(S,F_S(-1))=0$ so $F_S$ is $\mu$-stable.
			If $F$ is not locally free, then by Proposition~\ref{PropDescriptionSheafMx217}, $F$ lies in an exact sequence
			$$0 \to F \to E \to \calO_L \to 0$$
			with $E\in \calM_X(2,1,6)$ stable \textit{vector bundle} and $L\subset X$ a line. Restricting this sequence to $S$ (using Lemma~\ref{lemmahigherpullbacktorsionfree}) we get
			\begin{eqnarray}\label{EqnSeqEFOLrestrS}
				0 \to F_S \to E_S \to \calO_Z \to 0
			\end{eqnarray}
			with $Z$ a $0$-dimensional subscheme. By \cite[Prop. 3.4]{BrambillaFaenzig7}, $E$ satisfies the same vanishings as $F$, in particular $E$ is $\mu$-stable, and the previous arguments apply identically to prove that $E_S$ is also $\mu$-stable. This implies by (\ref{EqnSeqEFOLrestrS}) that $F_S$ is torsion free, and any destabilizing subsheaf of $F_S$ also destabilizes $E_S$, which is not possible, so $F_S$ is $\mu$-stable.
		\end{proof}

		Our goal is to study the restriction morphism $\mathsf{res}$ (\ref{EqnRESmorphism}). In a first part, we will prove that the restriction is generically injective, see Theorem \ref{thmresinj}.
		In a second part, we will identify the subspace of $\calM_X$ on which the restriction is not injective and study the image of $\calM_X$ in $\calM_S$, see Theorem~\ref{thmrestotal}.

		\subsection{Globally generated sheaves}\label{sectionglobalgen}

		Let us denote $\phi\coloneqq \phi_{11}, \phi^!\coloneqq\phi_{11}^!$ for simplicity. We consider a sheaf $F\in\calM_X$.
		By \cite[Lem. 4.3]{BrambillaFaenzig9}, there is an exact sequence
		\begin{eqnarray}\label{UFXexseq} 
			0 \to \calU^\vee \to \phi\phi^!F \to F \to 0.
		\end{eqnarray}
		As $F$ is torsion-free, in view of Lemma \ref{lemmahigherpullbacktorsionfree} the restriction to $S$ gives the exact sequence
		\begin{eqnarray}
			0 \to \calU^\vee_S \to i_{SX}^*\phi\phi^!F \to F_S \to 0. \label{UFSexseq}
		\end{eqnarray}
		
		\begin{proposition}\label{propdimextFS}
			For $F$ globally generated we have
			$$\dim\Ext^1(F_S,\calU_S^\vee)=1.$$
		\end{proposition}
		Note that by Theorem~\ref{ThmBFg9MXisBlowUpPic2} the general element of $\calM_X$ is globally generated.

		\begin{proof}
			We need to consider two exact sequences, namely
			\begin{eqnarray}
				&0& \to \calU \to V\otimes \calO_X \to \calU^\vee \to 0 \label{UUexseq}\\
				&0& \to \calU^\vee(-1) \to \calU^\vee \to \calU^\vee_S \to 0, \label{UXSexseq}
			\end{eqnarray}
			where $V$ is the $\matC$-vector space of dimension $6$ defining $LG(3,6)$. For the sequence (\ref{UUexseq}), note that the universal quotient bundle is isomorphic to the dual of the universal subbundle on $\Sigma$, thanks to the symplectic form on $V$.

			From Hirzebruch-Riemann-Roch, we can compute $\chi(F_S,\calU_S^\vee)= \chi(S,F_S\otimes\calU_S)$. We have $\ch(F_S)=(2,H_S,1)$, (\ref{UUexseq}) gives $\ch(\calU_S)=(3,-H_S,0)$ and $\td(S)=(1,0,2)$ as $S$ is a $K3$ surface. We obtain
			\begin{eqnarray*}
				\chi(S,F_S\otimes \calU_S) &=& \int (2,H_S,1)(3,-H_S,0)(1,0,2) =-1.
			\end{eqnarray*}
			To obtain the promised result, we will prove $\Ext^2(F_S,\calU_S^\vee) = 0 = \Hom(F_S,\calU_S^\vee)$.
			
			First, we have $\mu(F_S)=1/2$ and $\mu(\calU^\vee_S)=1/3$. By stability, we get $\Hom(F_S,\calU^\vee_S)=0$.

			\begin{lemma}\label{lemmahomUFS}
				For any $F\in\calM_X$, we have
				$$\Hom_S(\calU_S^\vee,F_S) \simeq \Hom_X(\calU^\vee,F)$$
			\end{lemma}
			
			\begin{proof}
				We proceed as explained in Remark~\ref{RmkComputationsExtgroupsTools}. By Serre duality on $X$ and $S$, the statement is equivalent to $\Ext^2_S(F_S,\calU_S^\vee) \simeq \Ext^3_X(F,\calU^\vee(-1))$. Apply $R\Hom_X(F,-)$ to (\ref{UXSexseq}) to get
				$$\Ext^2_X(F,\calU^\vee) \to \Ext^2_S(F_S,\calU_S^\vee) \to \Ext^3_X(F,\calU^\vee(-1)) \to \Ext^3_X(F,\calU^\vee).$$
				Now applying $R\Hom_X(F,-)$ to (\ref{UUexseq}) and since $H^k(X,F(-1))=0 \ \forall k$ (Proposition \ref{PropDescriptionSheafMx217}), we have
				$$\Ext^k_X(F,\calU^\vee) \simeq \Ext^{k+1}_X(F,\calU) \ \ \forall k.$$
				We have $\Ext^2_X(F,\calU^\vee) \simeq \Ext^3_X(F,\calU)=\Hom_X(\calU,F(-1))=0$ by stability since $\mu(F(-1))=-1/2$ and $\mu(\calU)=-1/3$. Moreover $\Ext^3_X(F,\calU^\vee)\simeq \Ext^4_X(F,\calU)=0$. Hence we obtain $\Ext^2_S(F_S,\calU_S^\vee) \simeq \Ext^3_X(F,\calU^\vee(-1))$.
			\end{proof}
			
			Now, for $F$ globally generated we have $\Ext^3_X(F,\calU^\vee(-1))\simeq \Hom_X(\calU^\vee,F)=0$ (Proposition \ref{PropDescriptionSheafMx217}), and hence by Lemma \ref{lemmahomUFS} we conclude the proof.
		\end{proof}
		
		From the exact sequence (\ref{UFSexseq}) we obtain the following corollary.
		
		\begin{corollary}\label{coroFSisomimpliesiphiphi!Fisom}
			Let $\widetilde{X}=\Sigma\cap \matP \widetilde{V}_{11}$, $\widetilde{V}_{11}\subset V_{14}$ be another Fano threefold constructed as in \S \ref{sectionsetup}. Let $\widetilde{\Gamma}$ be its associated quartic plane curve, and let $\widetilde{\phi}_{11}\colon \widetilde{\Gamma} \to \widetilde{X}$ be the functor obtained by HPD. For $F\in\calM_X$ and $\widetilde{F}\in\calM_{\widetilde{X}}$ globally generated, we have
			$$F_S\simeq \widetilde{F}_S \Rightarrow i_{SX}^*\phi\phi^!F \simeq i_{S\widetilde{X}}^*\widetilde{\phi}\widetilde{\phi}^{!}\widetilde{F}.$$
		\end{corollary}

		\begin{theorem}\label{thmresinj}
			The morphism $\textup{res}$ is injective on the set of globally generated sheaves (in particular, it is generically injective). Moreover, in the notation of Corollary  \ref{coroFSisomimpliesiphiphi!Fisom}, for any two globally generated sheaves $F\in\calM_{X},\widetilde{F}\in\calM_{\widetilde{X}}$, we have 
			$$F_S\simeq \widetilde{F}_S \Leftrightarrow X=\widetilde{X} \text{ and } F\simeq \widetilde{F}.$$
		\end{theorem}
		
		\begin{proof}\label{proofthmresinj}
			Let $F\in\calM_{X}$, $\widetilde{F}\in\calM_{\widetilde{X}}$ be globally generated sheaves over $X$ and $\widetilde{X}$ such that $F_S\simeq \widetilde{F}_S$. By Corollary \ref{coroFSisomimpliesiphiphi!Fisom} and Lemma \ref{lemisoway}, we obtain
			\begin{eqnarray}
				F_S \simeq \widetilde{F}_S &\Leftrightarrow & i_{SX}^*\phi\phi^!F \simeq i_{S\widetilde{X}}^*\widetilde{\phi}\widetilde{\phi}^!\widetilde{F} \nonumber \\
				&\Leftrightarrow &  \phi_{10} (i_{\Gamma S'})_* \circ\phi^!F \simeq  \phi_{10} (i_{\widetilde{\Gamma} S'})_* \circ \widetilde{\phi}^!\widetilde{F} \text{ by Lemma } \ref{lemisoway} \nonumber \\
				&\Leftrightarrow &  (i_{\Gamma S'})_* \circ \phi^!F \simeq   (i_{\widetilde{\Gamma} S'})_* \circ \widetilde{\phi}^!\widetilde{F} \ \text{ since } \phi_{10} \text{ is an equivalence. } \label{isotorsionlinebundle}
			\end{eqnarray}

			We know by \cite{BrambillaFaenzig9} that $\phi^!F$ and $\widetilde{\phi}^!\widetilde{F}$ are line bundles on $\Gamma$ and $\widetilde{\Gamma}$ respectively. But then $(i_{\Gamma S'})_* \circ \phi^!F$ and $(i_{\widetilde{\Gamma} S'})_* \circ \widetilde{\phi}^!\widetilde{F}$ are isomorphic torsion sheaves of rank one supported on a curve, hence $\Gamma=\widetilde{\Gamma}$. Since $\Gamma$ is given by the linear intersection $\Sigma^\vee\cap \matP L^\vee$ (see the beginning of \S \ref{sectionHPD}), we obtain $X=\widetilde{X}$ for $X=\Sigma\cap \matP L$. Finally (\ref{isotorsionlinebundle}) implies that $\phi^!F=\phi^!\widetilde{F}$ because the pushforward by closed immersion is fully faithful. Hence $F \simeq \widetilde{F}$ as $\phi^!\colon \calM_X \to \Pic^2(\Gamma)$ is injective on globally generated sheaves: these are exactly the sheaves which are not in the exceptional divisor of the blow-up $\calM_X \to \Pic^2\Gamma$.
		\end{proof}

		\subsection{The non-injectivity locus}\label{sectionresnotinj}
		
		We investigate now which sheaves on $\calM_X$ have the same restriction on $\calM_S$. Recall from Proposition~\ref{PropDescriptionSheafMx217} that a sheaf $F$ which is not locally free (resp. not globally generated) is so along a line $M_F$ (resp. $L_F$). 
		Note that the exact sequence (\ref{EqFnotggIFOL}) is induced by the evaluation map $\mathsf{ev}_F\colon H^0(F)\otimes \calO_X \to F$, that is $\calO_{L_F}(-1)=\coker(ev_F)$.
		
		\begin{proposition}\label{PropFGnotlfggMFisLG}
			Let $F\not\simeq G\in \calM_X$ such that $F_S\simeq G_S$. Then both $F$ and $G$ are not locally free nor globally generated. Moreover $L_F=M_G$ and $M_F=L_G$ but $M_F\neq L_F$.
		\end{proposition}
		
		\begin{proof}
			If $F$ were globally generated, combining Proposition~\ref{PropDescriptionSheafMx217} and Lemma~\ref{lemmahomUFS} we would get $0=\Hom_X(\calU^\vee, F)\simeq \Hom_S(\calU^\vee_S, F_S) \simeq \Hom_X(\calU^\vee, G)=0$, in particular $G$ would be globally generated, contradicting Theorem~\ref{thmresinj}.

			To prove that neither $F$ nor $G$ is locally free, it is enough to prove it for $G$. Indeed, in this case the restriction of (\ref{EqFnotlocfreeFF**OM}) to $S$ gives an exact sequence
			$$0 \to G_S \to (G^{**})_S \to \calO_y \to 0$$
			with $y=M\cap S$ (recall we assumed that $S$ does not contain a line). Applying $R\calH om(-,\calO_S)$ to the sequence, and using $\calH om(\calO_y,\calO_S)=0$ gives $G_S^*\simeq (G^{**})_S^*$, and therefore $G_S^{**}\simeq (G^{**})_S^{**}\simeq (G^{**})_S$. So $G_S$ is not locally free and $F_S \simeq G_S$ implies that $F$ is not locally free neither.

			Let's prove that $G$ is not locally free and $M_G=L_F$. By symmetry, this gives $M_F=L_G$ aswell. Apply $\Hom(F,-)$ to $0 \to G(-1) \to G \to G_S \to 0$ to get
			\begin{eqnarray}\label{FGF_SG_S}
				0 \to \Hom(F,G) \to \Hom(F_S,G_S) \to \Ext^1(F,G(-1)).
			\end{eqnarray}
			We get $\Ext^1(F,G(-1))\simeq \Ext^2(G,F) \neq 0$, otherwise an isomorphism $F_S\simeq G_S$ would lift to an isomorphism $F\simeq G$ by exactness of (\ref{FGF_SG_S}).
			Apply $\Hom(G,-)$ to (\ref{EqFnotggIFOL}) to obtain
			$$\Ext^2(G,I) \to \Ext^2(G,F) \to \Ext^2(G,\calO_{L_F}(-1)) \to \Ext^3(G,I).$$
			
			\begin{itemize}
				\item Apply $\Hom(G,-)$ to  $0 \to \calO_X \to \calU^\vee \to I \to 0$ to get
				$$\Ext^2(G,\calU^\vee) \to \Ext^2(G,I) \to \Ext^3(G,\calO_X).$$
				But $\Ext^3(G,\calO_X)=0$ as $H^k(X,G(-1))=0 \ \ \forall k$ (Proposition \ref{PropDescriptionSheafMx217}), and from (\ref{UUexseq}) $\Ext^2(G,\calU^\vee) \simeq \Ext^3(G,\calU) \simeq \Hom(\calU, G(-1)) = 0$ by stability and comparing slopes. Hence $\Ext^2(G,I)=0$.
				
				\item $\Ext^3(G,I) \simeq \Hom(I,G(-1)) = 0$ comparing slope and by stability.
			\end{itemize}
			We obtain $\Ext^2(G,F) \simeq \Ext^2(G,\calO_{L_F}(-1))$. Hence for $F\not\simeq G$ with $F_S\simeq G_S$, we have $\Ext^2(G,\calO_{L_F}(-1)) \simeq \Ext^1(\calO_{L_F},G)^* \neq 0$. Therefore we have a non-trivial exact sequence
			\begin{eqnarray}\label{exseqGnotlocfree}
				0 \to G \to \calG \to \calO_{L_F} \to 0,
			\end{eqnarray}
			and computations of Chern classes gives
			
			\begin{center}
				\begin{tabular}{ l c c c }
					& $G$ & $\calG$  &  $\calO_{L_F}$ \\
					$\rk$ & $2$ &  $2$    & $0$ \\
					$c_1$ & $1$ & $1$ & $0$\\
					$c_2$ & $7$ & $6$ & $-1$\\
					$c_3$ & $0$ & $0$ & \phantom{.}$1$.\\
				\end{tabular}
			\end{center}
			If $\calG$ is not torsion free, consider its torsion subsheaf $\calG_t$ and the exact sequence
			$$0 \to \calG_t \to \calG \to \calG_f \to 0.$$
			First note that $\calG_f$ is stable. Indeed, if $K\subseteq \calG_f$ with rank $1$ and $c_1(K)=c\ge 1$, denote $K''$ the image of $K$ in $T$. We can consider $0 \to K' \to K \to K'' \to 0$ and computing Chern classes we obtain $c_1(K')=c \ge 1$, but $K'\subseteq G$ as it is the kernel of $K \to K''$, so $K'$ destabilizes $G$ which is absurd.
			
			The composition $G \hookrightarrow \calG \to \calG_f$ is still injective as $G$ is torsion free. 
			Recall that if a sheaf $\calH$ is supported on an integral subvariety $Z$ of codimension $m$, and has rank $r$ at a generic point of $Z$, Grothendieck-Riemann-Roch implies that $c_k(\calH)=0$ for $1\leq k \leq m-1$ and:
			$$c_m(\calH)=(-1)^{m-1}r[Z].$$
			If $Z$ is reducible, the same formula holds by addition over the components of $Z$ having maximal dimension. Using that $\calG_f$ is stable, a quick computation gives $c_1(\calG_t)=0$ (hence $c_1(\calG_f)=1$). In particular, we obtain that $G$ is not locally free by \cite[Lem. 2.1]{BrambillaFaenzig9}. Another computation gives $c_2(\calG_t)=-1$ or $0$, hence $c_2(\calG_f)=6$ or $7$. We distinguish these two cases. Set $d=c_3(\calG_t)$.
			
			\begin{enumerate}
				\item If $c_2(\calG_f)=7$, we obtain $c_3(\calG_f)=1-d$. The quotient of the injective map $G \hookrightarrow \calG_f$ is a zero dimensional torsion sheaf $T$ with $c_3(T)=1-d$, hence $1-d\geq 0$. From Proposition~\ref{PropDescriptionSheafMx217}, either $\calG_f$ is locally free or there is an exact sequence 
				$$0 \to \calG_f \to E \to \calO_L \to 0$$
				with $E$ rank $2$ vector bundle with $c_1(E)=1$, $c_2(E)=6$. In the latter case, computation of Chern classes gives $1-d=0$, hence $\calG_f\simeq G$. But the map $\calG \to \calG_f \simeq G$ splits the exact sequence (\ref{exseqGnotlocfree}) which is absurd.
				In the former case ($\calG_f$ locally free), the inclusion $G \hookrightarrow \calG_f$ implies that $G$ is locally free on an open subset of codimension $3$, which is absurd as the locus of non locally freeness of $G$ is the line $M_G$. We conclude that $c_2(\calG_t)\neq -1$.
				\item If $c_2(\calG_f)=6$, consider the exact sequence 
				$$0 \to \calG_f \to \calG_f^{**} \to Q \to 0.$$
				We know that $\calG_f^{**}$ is stable (from the same proof as for $\calG_f$) and satisfies $c_1(\calG_f^{**})=1$. From \cite[Lem. 3.1]{BrambillaFaenzig7} we must have $c_2(\calG_f^{**})\geq 6$, hence $c_2(\calG_f^{**})= 6$ and $Q$ is $0$-dimensional. Moreover, since $\calG_f^{**}$ is reflexive it also satisfies $c_3(\calG_f^{**})\geq 0$ (generalization of \cite[Prop. 2.6]{HartshorneStableReflexiveSheaves}). From Proposition~\ref{PropDescriptionSheafMx217} again, $\calG_f^{**}$ must be locally free. We deduce that $\calG_f$ is locally free on an open subset $U\subset X$ of codimension $3$.
				
				The cokernel of the injective map $G \hookrightarrow \calG_f$ is a torsion sheaf $T$ with $c_2(T)=-1$, so it is supported on a line $L$. In particular, we obtain that $G$ is locally free on $U\smallsetminus L$, so $L=M_G$. The composition $\calG \twoheadrightarrow \calG_f \twoheadrightarrow T$ factors through $\calO_{L_F}$ as shown on the commutative diagram
				
				\begin{center}
					\begin{tikzcd}
						0 \arrow[r]& G \arrow[d, "="] \arrow[r] & \calG \arrow[r]\arrow[d, two heads]  & \calO_{L_F} \arrow[r]\arrow[d, dashed, two heads] & 0\\
						0 \arrow[r] & G \arrow[r] & \calG_f \arrow[r] & T \arrow[r] & 0.
					\end{tikzcd}
				\end{center}
				We obtain a surjective map $\calO_{L_F} \twoheadrightarrow T$, which gives $L_F=M_G$.
			\end{enumerate}
			
		\end{proof}
		
		The following results is proved in \cite[arXiv v.1, Lem. 5.6]{BrambillaFaenzig9}, but we add a more direct proof here.
		
		\begin{lemma}\label{LemmaLFisnotMF}
			In the notation of Proposition~\ref{PropFGnotlfggMFisLG}, we have $M_F\neq M_G$.
		\end{lemma}
		
		\begin{proof}
			Assume $M_F=M_G$ (equivalently, $M_F=L_F$) for the seek of contradiction. Denote $I_F$ and $I_G$ the sheaves appearing in (\ref{EqFnotggIFOL}) and $\gamma_F,\gamma_G\in \Hom(\calO_X,\calU^\vee)$ the map appearing in (\ref{EqFnotggOUI}) with respect to $F$ and $G$ respectively. Recall that $I_F$ is the image of the natural evaluation map $\mathsf{ev}_{X,F}\colon H^0(X,F)\otimes \calO_X \to F$, and similarly for $G$. Hence $(I_F)|_S=(\im(\mathsf{ev}_{X,F}))_S\simeq \im(\mathsf{ev}_{S,F_S})$, and since $F_S\simeq G_S$ we get $I_S\coloneqq (I_F)|_S \simeq (I_G)|_S$. Therefore$(\gamma_F)|_S$ and $(\gamma_G)|_S$ are proportional. But from
			$$0 \to \calU(-1) \to \calU \to \calU_S \to 0$$
			we see that $H^0(X,\calU) \to H^0(S,\calU_S)$ is injective, hence $\gamma_F$ and $\gamma_G$ are proportional, that is $I\coloneqq I_F\simeq I_G$.
			But we have
			\begin{eqnarray*}
				\Ext^1_X((i_{LX})_*\calO_{L}(-1),I) &\simeq& \Ext^2_X(I,(i_{LX})_*\calO_{L}(-2))^* \\
				&\simeq & \Ext^2_L(Li_{LX}^*I,\calO_{L}(-2))^* \\
				&\simeq & \Ext^{-1}_L(\calO_L,Li_{LX}^*I)^* \\
				&\simeq & \Ext^{-1}_L(\calO_L,Li_{LX}^*I)^* 
			\end{eqnarray*}
			We know that the latter group is not trivial, because $F\in\Ext^1_X((i_{LX})_*\calO_{L}(-1),I)$ is torsionfree. Since $L$ is a curve, $Li_{LX}^*I\simeq \oplus_k L^ki_{LX}^*I[-k]$. We will prove $L^ki_{LX}^*I=0$ for $k>1$ and $L^1i_{LX}^*I\simeq \calO_L$, i.e. $F\simeq G$ which is absurd.
			
			Apply $Li^*_{LX}$ to (\ref{EqFnotggOUI}) to get the exact sequence
			$$0 \to L^1i^*_{LX}I \to \calO_L \xrightarrow{k} \calU_L^\vee \to I_L \to 0.$$
			We obtain $L^ki_{LX}^*I=0$ for $k>1$. Then $L^1i^*_{LX}=\ker(k)$ and denote $P=\im(k)$. We know that $\rk(I_L)\geq 2$ by upper semicontinuity. If $\rk(I_L)=2$, then $\rk(P)=1$ and hence $L^1i^*_{LX}I=0$ because it is a torsion subsheaf of $\calO_L$. But this would give $\Ext^{-1}_L(\calO_L,Li_{LX}^*I)^*=0$ which is absurd. Since $\calU^\vee_L\to I_L$ is surjective, we get $\rk(I_L)=3$, hence $P=0$ and $L^1i^*_{LX}I\simeq \calO_L$.
		\end{proof}

		\begin{theorem}\label{thmrestotal}
			For X general, the image of $\calM_X$ in $\calM_S$ via the restriction map is a Lagrangian subvariety with finitely many singular points, each of which have exactly $2$ preimages in $\calM_X$.
		\end{theorem}

		\begin{proof}
			The assumption $X$ general ensures that $\calM_X$ is smooth (recall that $\calM_X$ is the blow-up of $\Pic^2(\Gamma)$ along a subscheme isomorphic to the Fano of lines $\calH^0_1(X)$ which is smooth for $X$ general \cite[Thm. 4.2.7]{IskovskikhProkhorovFanoVarieties}). 
			
			\begin{lemma}
				For all $F\in\calM_X$, we have $\Ext^2(F,F)=0$.
			\end{lemma}
			
			\begin{proof}
				Apply $R\Hom(-,F)$ to (\ref{UFXexseq}) to get
				\begin{eqnarray*}
					0 \to \Hom(\calU^\vee,F) \to \Ext^1(F,F) \xrightarrow{t} \Ext^1(\phi^!F,\phi^!F)  
					\to \Ext^1(\calU^\vee,F) \\ \to  \Ext^2(F,F) \to 0 
				\end{eqnarray*}
				and $\Ext^2(\calU^\vee,F)=0$. The map $t$ is the tangent map of $\phi^!$, which is an isomorphism (resp. has rank $2$) when $F$ is (resp. is not) globally generated. We want to prove
				\begin{eqnarray}\label{EqExt1UFisCokerTangMap}
					\Ext^1(\calU^\vee,F)=\coker(t).
				\end{eqnarray} 
				Note that $\chi_F\coloneqq \chi(\calU^\vee,F)=\chi(X,F\otimes \calU)$ only depends on the Chern classes of $F$ thanks to Hirzebruch-Riemann-Roch, hence $\chi_F$ is constant on $\calM_X$. Therefore if (\ref{EqExt1UFisCokerTangMap}) holds for one sheaf in $\calM_X$ then it holds for all of them. By \cite[\S 4]{BrambillaFaenzig9} we know that there exists a sheaf $G\in\calM_X$ with $\Ext^2(G,G)=0$, which gives (\ref{EqExt1UFisCokerTangMap}).
			\end{proof}

			First, from \cite{BeauvilleFanoThreefoldK3} we know that $\mathsf{res}\colon \calM_X \to \calM_S$ induces an immersion (that is, a morphism with injective differential) of $\calM_X$ onto a Lagrangian subvariety of $\calM_S$. Moreover, by Theorem~\ref{thmresinj} it is injective on the set of sheaves which are either locally free or globally generated. 
			
			Consider a singular point of $\mathsf{res}(\calM_X)$, it corresponds to the image of a sheaf $F\in\calM_X$ which is neither globally generated nor locally free. But from Proposition \ref{PropFGnotlfggMFisLG}, the set of sheaves $E$ with $E_S\simeq F_S$ consists exactly in $\{F,G\}$ with $G$ the sheaf for which $M_G=L_F$ and $L_G=M_F$, and these two (distinct by Lemma~\ref{LemmaLFisnotMF}) lines intersect on $S$.
			From \cite[\S 4.2]{IskovskikhProkhorovFanoVarieties}  each line in $X$ intersects a finite number of lines. In particular, the scheme parametrizing couples of intersecting lines in $X$ has dimension $1$, and the image of the intersection points of such couples of lines forms a $1$-dimensional subscheme of $X$. This subscheme intersects the general divisor $S$ in a finite number of points. 
		\end{proof}

		\begin{remark}
			It would be interesting to know if these singular points are ordinary double points. This amounts to prove that two non-trivial extensions $T_F\in\Ext^1(F,F)$ and $T_G\in\Ext^1(G,G)$ for $F,G\in\calM_X$ such that $F_S\simeq G_S$ do not restrict by $i_{SX}^*$ to the same extension in $\Ext^1(F_S,F_S)$.
		\end{remark}

		\section{Lagrangian fibrations and birational models}\label{sectionLagrangianFibration}

		\subsection{Relative moduli spaces}\label{sectionLagFibBirModsetup}

		In this section, we fix $S$ and we vary $X$. Recall from \S \ref{setupFanog9} that we defined the Fano $X$ (resp. the K3 surface $S$) as the intersection of the Lagrangian Grassmannian $\Sigma\subset \matP V_{14}$ with some linear subspace $\matP V_{11}$ (resp. $\matP V_{10}$) for a $11$- (resp. $10$-) dimensional vector subspace of $V_{14}$. 
		
		For now on, we fix $S$, i.e. we fix $V_{10}\subset V_{14}$. The Fanos $X$ containing $S$ are parametrized by the set of $11$-dimensional vector subspaces $W$ verifying $V_{10}\subset W \subset V_{14}$, so they are parametrized by $\matP^3 = \matP(V_{14}/V_{10})$. Dually, the set of corresponding plane curves in $(\matP^{13})^\vee$ (defined by $\Sigma^\vee \cap W^\perp$, see \S \ref{sectionHPD}) are parametrized by the $3$-dimensional vector subspaces $W^\perp \subset V_{10}^\perp$, hence parametrized by the same projective space $\matP^3=(\matP V_{10}^\perp)^\vee$. 
		
		Let us introduce the following notations:

		
		\begin{itemize}[leftmargin=*]
			\item Let $\calW\subset \matP^3$ be the subset of vector subspaces $W$ such that the Fano (resp. plane curve) $\Sigma \cap \matP W$ (resp. $\Sigma^\vee \cap \matP W^\perp$) is smooth. We obtain two smooth morphisms
			
			\begin{center}
				\begin{tikzcd}
					{\mathfrak{X}} \ar[r,hook] \ar[d] & \Sigma \times \matP^3 \ar[d] & {\mathfrak{G}} \ar[d] \ar[r,hook] & \Sigma^\vee \times \matP^3 \ar[d] . \\
					\calW \ar[r,hook] & \matP^3 & \calW \ar[r,hook] & \matP^3
				\end{tikzcd}
			\end{center}	
			such that $\frak{X}_{W}=\Sigma\cap \matP W$ and $\frak{G}_{W} = \Sigma^\vee \cap \matP W^\perp$ for each $W\in\calW$. To simplify notations, we will simply write $[X],[\Gamma]\in \calW$ for Fanos and plane curves parametrized by $\calW$. Up to shrinking $\calW$ a little bit, we can assume that $\calM_{X}$ is smooth for all $[X]\in\calW$ (see proof of Theorem \ref{thmrestotal}).
			
			\item $\calM_\mathfrak{X}\coloneqq \calM_{\mathfrak{X}/\calW}(2,1,7)$ stands for the relative moduli space of sheaves of the family $\mathfrak{X} \to \calW$. The fibres satisfy $(\calM_{\frak{X}})_{[X]}\simeq \calM_X$ for any $X\in\calW$.
			
			\item $\calM_{(S',\alpha)}\coloneqq \calM_{(S',\alpha)}[0,H',0]$ stands for the moduli space of $\alpha$-\textit{twisted} torsion sheaves on $S'$ supported on curves on the primitive polarization $H'$ (see \cite{YoshiokaModuliTwistedSheaves} for the existence of moduli spaces of twisted sheaves). This space admits a map $p\colon \calM_{(S',\alpha)}[0,H',0]\to \matP^3=|H'|$ defined in \cite{BeauvilleSystemesHamiltoniens} called \textit{Beauville-Mukai integrable system}, sending a sheaf to its support.
			\item $\Pic^2_\alpha(\mathfrak{G})\coloneqq p^{-1}(\calW)\subset \calM_{(S',\alpha)}$ stands for the open subspace of smooth fibres. For any curve $[\Gamma]\in \calW$, we have $p^{-1}([\Gamma])\simeq \Pic^2(\Gamma)$, which motivates the notation.
			
		\end{itemize}

		\begin{remark}
				In fact, the definition of $\calM_{(S',\alpha)}$ (or more precisely, the last digit in the Mukai vector $[0,H',0]$) requires an extra choice, for instance the choice of a B-field representing $\alpha$, or a lift of $\alpha$ to the Special Brauer group of $S$ (see \cite{HuybMatteiSBr}). We will omit this, because this choice is of no importance in this text.

				The Brauer class $\alpha$ does not play a role when one considers a single curve $[\Gamma]\in \calW$ (because the Brauer group of a curve over an algebraically closed field is trivial), but it does when considering the whole family $\mathfrak{G}$. In general, the twisted and non-twisted moduli spaces are not isomorphic, nor birational. For instance, it is known that the movable and nef cones of $\calM_{S'}[0,H',0]$ coincide \cite[Lem. 3.12]{FMOGSGeomAntisympInvolI} (that is, this moduli space admits no other birational model), which we will disprove in the twisted case in Proposition~\ref{PropPosMovNefMS}. 
				
			\end{remark}

			Since both $\calM_{\mathfrak{X}}$ and $\Pic^2_\alpha(\mathfrak{G})$ are nonsingular, they are both flat hence smooth over $\calW$ by miracle flatness \cite[Ex. 10.9]{HartshorneAlgebraicGeometry}.
			Using the next lemma (Lemma \ref{lemmaphiphi!stable}), in view of Remark \ref{rmkfamilyHPD}, one can prove using the machinery developped in \cite[\S 2]{KuznetsovHyperplaneSections} that the functors introduced \S \ref{sectionStudyMxandrestriction} can be defined relatively to give morphisms
			\begin{eqnarray*}
				\Phi^! &\colon& \calM_{\mathfrak{X}} \to \Pic^2_\alpha(\mathfrak{G}),  \\
				\Phi &\colon&  \Pic^2_\alpha(\mathfrak{G})  \to \calM_{\mathfrak{X}/\calW}(5,2,31,17).
			\end{eqnarray*}
			
			The value $(5,2,31,17)$ can be computed using (\ref{UFXexseq}). We will need the next very useful criterion for flatness.
			
			\begin{proposition}[Crit\`ere de platitude par fibres, \cite{stacks-project}, \href{https://stacks.math.columbia.edu/tag/039A}{Tag 039A}]\label{PropCriterePlatitudeParFibres}
				Let $S$ be a scheme, let $R\to T$ be a morphism of scheme over $S$. Assume that 
				\begin{itemize}
					\item $R$ is flat over $S$, 
					\item $f_s\colon R_s \to T_s$ is flat for every $s\in S$.
				\end{itemize}
				Then $f$ is flat.
			\end{proposition}

			First, fix $[X]\in \calW$ a smooth Fano and $\Gamma$ the corresponding curve.
			Recall the definition of the functors $\phi\coloneqq \phi_{11}, \phi^!\coloneqq\phi_{11}^!$ in \S \ref{sectionnotationXGammaHPD}. Consider the mutation functor
			\begin{eqnarray}\label{EqnPhiPhinotDbXtoDbX}
				\phi\phi^!\colon \D^b(X) \to \phi\D^b(\Gamma)\subset \D^b(X).
			\end{eqnarray}

			\begin{lemma}\label{lemmaphiphi!stable}
				Let $F\in \calM_X$ be a sheaf. Then the sheaf $\phi\phi^!F$ and $(\phi\phi^!F)_S$ are $\mu$-stable sheaves.
			\end{lemma}
			
			\begin{proof}
				Since the restriction to $S$ of any $\mu$-destabilizing subsheaf of $\phi\phi^!F$ would destabilize $(\phi\phi^!F)_S$, it is enough to prove that $(\phi\phi^!F)_S$ is $\mu$-stable. This is done in \cite[Lem. 2.1]{YoshiokaSomeExamplesMukaiReflectionsK3Surf}.
			\end{proof}

			Since $\calM_X$ is irreducible, there is an irreducible component $M\subset \calM_X(5,2,31,17)$ with $\phi\phi^!\calM_X \subset M$.
			
			\begin{lemma}
				For any $F\in\calM_X$, the space $M$ is smooth at $[\phi \phi^!F]$ and $T_{[\phi \phi^!F]}M$ has dimension $3$.
			\end{lemma}
			
			\begin{proof}
				We have
				\begin{eqnarray*}
					T_{[\phi \phi^!F]}M &\simeq & \Ext^1_X(\phi\phi^!F,\phi\phi^!F) \\
					&\simeq & \Ext^1_\Gamma(\phi^!F, \phi^!F) \\
					&\simeq & \matC^{g(\Gamma)}=\matC^3
				\end{eqnarray*}
				because $\phi$ is fully faithful, and $\phi^!F$ is a line bundle on $\Gamma$ which is a curve of genus $3$.
				Similarly, $\Ext^2_X(\phi\phi^!F,\phi\phi^!F) \simeq \Ext^2_\Gamma(\phi^!F,\phi^!F) = 0$, so the obstruction space vanishes and $M$ is smooth at $[\phi\phi^!F]$.
			\end{proof}
			
			\begin{proposition}\label{propPic2isoMx52}
				We have an isomorphism $\Pic^2(\Gamma)\simeq M$. Moreover, these isomorphisms can be defined relatively to give an isomorphism of $\Pic^2_\alpha(\mathfrak{G})$ onto an irreducible component of $\calM_{\mathfrak{X}/\calW}(5,2,31,17)$.
			\end{proposition}

			\begin{proof}
				Since $\phi$ is fully faithful, the morphism $L \in \Pic^2\Gamma \mapsto \phi L\in M$ is both injective and a local isomorphism as the induced linear map $\Ext^1(L,L) \to \Ext^1(\phi L,\phi L)$ is an isomorphism. Moreover $\Pic^2\Gamma$ and $M$ are irreducible with same dimension so the morphism is also surjective.
				
				Now, consider the morphism $\Phi\colon \Pic^2_\alpha(\mathfrak{G}) \to \calM_{\mathfrak{X}/\calW}(5,2,31,17)$ over $\calW$. Recall that we assumed that both spaces are flat over $\calW$. On each fibre over a closed point $w\in\calW$, the morphism $\Phi_w$ is an isomorphism. By Proposition \ref{PropCriterePlatitudeParFibres}, we obtain that $\Phi$ is flat. Since $\Pic^2_\alpha(\mathfrak{G})$ is smooth, we obtain that $\Phi$ is smooth of relative dimension $0$, therefore \'etale, and since it is injective it must be an open immersion. 
				
				Moreover $\Pic^2_\alpha(\mathfrak{G})\to \calW$ is projective, hence the image $\im(\Phi)\subset \calM_{\mathfrak{X}/\calW}(5,2,31,17)$ is projective over $\calW$. In particular, $\im(\Phi)$ is universally closed, so the map
				$$\im(\Phi) \times_\calW \calM_{\mathfrak{X}/\calW}(5,2,31,17) \simeq \im(\Phi) \to \calM_{\mathfrak{X}/\calW}(5,2,31,17) $$
				is closed, and we obtain that $\im(\Phi)$ is a closed subset, and thus an irreducible component, of $\calM_{\mathfrak{X}/\calW}(5,2,31,17)$.
			\end{proof}
			
			In the following, we identify $\Pic^2_\alpha(\mathfrak{G})$ with the corresponding irreducible component of $\calM_{\mathfrak{X}/\calW}(5,2,31,17)$.

			\subsection{Global restriction to $S$}

			\begin{proposition}\label{propglobalresopenimm}
				The restrictions of sheaves from $[X]\in \calW$ to $S$ give relative restriction morphisms
				\begin{eqnarray*}
					&\sf{res}_\calW\colon & \calM_{\mathfrak{X}} \to  \calM_S, \\
					&\sf{res}'_\calW\colon & \Pic^2_\alpha(\mathfrak{G}) \to  \calM_S[5,2,6].
				\end{eqnarray*}
				
			\end{proposition}
			
			Note that the passage from $(5,2,31)$ to $[5,2,6]$ is just a change of notation, replacing Chern classes by the Mukai vector, see \S \ref{SectionCohomModuliSheaves}.
			
			\begin{proof}
				The embedding $S\times \calW \hookrightarrow \Sigma \times \Gr(11,V_{14})$ factors through $\mathfrak{X}$, that is we have an embedding
				$$S\times \calW \xhookrightarrow{j} \mathfrak{X}$$
				which is a morphism over $\calW$. 
				
				
				Consider the moduli functors $\mathbf{M}_{\mathfrak{X}/\calW}\coloneqq\mathbf{M}_{\mathfrak{X}/\calW}(2,1,7)$ and $\mathbf{M}_{S\times\calW/\calW}\coloneqq\mathbf{M}_{S\times\calW/\calW}(2,1,7)$ for the corresponding moduli problems (we use the bold notation to avoid confusion with moduli spaces). The pullback by $j$ gives a natural transformation
				$$\mathbf{j}^*\colon \mathbf{M}_{\mathfrak{X}/\calW} \to \mathbf{M}_{S\times\calW/\calW}.$$
				Both functors admit coarse moduli spaces $\calM_{\mathfrak{X}/\calW}$ and $\calM_{S\times\calW/\calW}$, hence we obtain a morphism
				$$j^*\colon \calM_{\mathfrak{X}/\calW} \to \calM_{S\times\calW/\calW}.$$
				Finally, we use the natural projection $\calM_{S\times\calW/\calW} \simeq \calM_S\times\calW \to \calM_S$ (in other words, we "forget" from which Fano $X$ a sheaf on $S$ comes from), and we obtain the desired morphism
				$${\sf res}_\calW\colon \calM_{\mathfrak{X}/\calW} \to \calM_S.$$
				In view of Proposition~\ref{propPic2isoMx52}, the same argument gives a morphism 
				$${\sf res}'_\calW\colon \Pic^2_\alpha(\mathfrak{G}) \to \calM_S[5,2,6].$$
			\end{proof}

			Consider
			\begin{itemize}
				\item $\calM_{\mathfrak{X}}^o\subset \calM_{\mathfrak{X}}$ the subset of globally generated sheaves,
				\item $\Pic^2_\alpha(\mathfrak{G})^o=\Phi^!(\calM_{\mathfrak{X}}^o)$.
			\end{itemize}

			\begin{lemma}
				The subspace $\calM_{\mathfrak{X}}^o\subset \calM_{\mathfrak{X}}$ and $\Pic^2_\alpha(\mathfrak{G})^o$ are open.
			\end{lemma}
			
			\begin{proof}
				Let us recall some facts about the construction of moduli spaces of sheaves (see \cite[I.4]{HuybrechtsLehnModuliofsheaves}). Here we use that semistable sheaves are stable in our case (Lemma \ref{LemmaStableSheavesAreMuStableAndRestriction}). There is an open subscheme
				$$\calR \subset \Quot_{\mathfrak{X}/\calW}(\calH)$$
				over $\calW$, where $\Quot_{\mathfrak{X}/\calW}(\calH)$ is a Quot scheme, parametrizing quotients $\calH_w \to F_w$ with $F_w\in\calM_{\mathfrak{X}_w}$, $w\in\calW$. Here, $\calH=\calO_{\mathfrak{X}}(-m)^{\oplus N}$ for some integers $m,N\geq 0$. Moreover, the relative moduli space of sheaves is constructed as a $\SL_N(\matC)$-GIT quotient of $\calR$, in particular the map (over $\calW$)
				$$\pi\colon \calR \twoheadrightarrow \calM_{\mathfrak{X}}$$
				is an open map, so we are reduced to prove that $\pi^{-1}(\calM_{\mathfrak{X}}^o)$ is open.
				
				Note that $\Quot_{\mathfrak{X}/\calW}(\calH)$ is a fine moduli space, in particular it carries a universal quotient family. Restricting it to $\calR$, we obtain a universal quotient family
				$$\rho\colon \calO_\calR\boxtimes \calH \to \calF$$
				on $\calR\times_{\calW}\mathfrak{X}$. Note that $\calF$ is $\calR$-flat by definition of the $\Quot$-scheme and since open immersions are flat morphisms.
				Any sheaf $F_w\in\calM_{\mathfrak{X}_w}$ with a given surjective map $\rho_w\colon \calH_w \twoheadrightarrow F_w$ is the pullback of $\rho$ by the base change $\Spec\matC \to \calR, *\mapsto [\rho_w]$.
				
				Denote $p_\calR,p_\mathfrak{X}$ the natural projection from $\calR\times_\calW \mathfrak{X}$.
				Consider the bundle $\calU_\mathfrak{X}$ obtained by pullback of $\calU_{\Sigma}$ by the composition
				$$\mathfrak{X} \hookrightarrow \Sigma \times \calW \twoheadrightarrow \Sigma.$$
				It is easy to see that $(\calU_{\mathfrak{X}})_w=\calU_{\mathfrak{X}_w}$ for any $w\in\calW$. We can thus consider
				$$\widetilde{\calF}\coloneqq\calF\otimes p_\mathfrak{X}^*\calU_\mathfrak{X}.$$
				For any $[\rho_w\colon \calH_w \to F_w]\in \calR_w$, the sheaf $F_w$ on $\mathfrak{X}_w$ is globally generated if and only if $H^0(\mathfrak{X}_w,F_w\otimes \calU_{\mathfrak{X}_w})=0$ (Proposition \ref{PropDescriptionSheafMx217}). But this is equivalent to 
				\begin{eqnarray}\label{H0vanishUnivQuotientFamily}
					H^0(\mathfrak{X}_w,\widetilde{\calF}_{[\rho_w]})=0. 
				\end{eqnarray}
				The subset $\calR^0\subset \calR$ where (\ref{H0vanishUnivQuotientFamily}) holds is open in $\calR$ by the semicontinuity theorem  \cite[III.12.8]{HartshorneAlgebraicGeometry}. Since $\calR^0=\pi^{-1}(\calM_{\mathfrak{X}})^o$, we conclude.
				
				Finally, the map
				$$\Phi^{!o}\colon \calM_{\mathfrak{X}}^o \to \Pic^2_\alpha(\mathfrak{G})$$
				is fibrewise an open immersion (see \S \ref{sectionStudyMxandrestriction}). Using again Proposition  \ref{PropCriterePlatitudeParFibres} with a similar argument as in Proposition \ref{propPic2isoMx52} we obtain that $\Phi^{!o}$ is an isomorphism onto its image. In particular, $\Pic^2_\alpha(\mathfrak{G})^o$ is open. This subset consists only in elements in $\Pic^2(\Gamma)$, $[\Gamma]\in\calW$, which are not in the locus blown up by $\phi^!$.
			\end{proof}

			\begin{proposition}\label{PropResMorphGiveOpenImmersion}
				The restriction morphisms in Proposition~\ref{propglobalresopenimm} restrict
				to open immersions
				\begin{eqnarray*}
					&\mathsf{res}_\calW^o\colon & \calM_{\mathfrak{X}}^o \hookrightarrow  \calM_S, \\
					&\mathsf{res}_\calW'^o\colon & \Pic^2_\alpha(\mathfrak{G})^o \hookrightarrow  \calM_S[5,2,6].
				\end{eqnarray*}
			\end{proposition}

			\begin{proof}

				Both $\calM_{\mathfrak{X}}^o$ and $\calM_S$ are smooth, hence $\mathsf{res}_\calW$ is smooth of relative dimension $0$, hence étale, and since $\mathsf{res}_\calW^o$ is injective (Theorem~\ref{thmresinj}) it must be an open immersion. The same argument applies to $\mathsf{res}_\calW'^o$.
				
			\end{proof}

			\begin{corollary}\label{corratlagfib}
				The morphism $\calM_S^o \to \calW$ which sends a sheaf of the form $F_S\in\calM_S$ with $F\in\calM_X$ globally generated to $[X]\in\calW$ gives a rational Lagrangian fibration
				$$\calM_S \dashrightarrow \matP^3,$$
				where the fibre over a point $[X]\in\calW$ is the open subset $\calM_X^o\subset \calM_X$ of globally generated sheaves.
			\end{corollary}
			
			\begin{proof}
				From Proposition~\ref{PropResMorphGiveOpenImmersion}, the sheaves of the form $F_S\in\calM_S$ with $F\in\calM_X$ for some $[X]\in\calW$ and $F$ globally generated form an open subset $\calM_{\mathfrak{X}}^o\subset \calM_S$, and from Theorem~\ref{thmresinj} such a sheaf cannot belong to $\calM_Y$ with $[X]\neq [Y]\in\calW$, hence the map $\calM_S \supset \calM_{\mathfrak{X}}^o \to \calW$ is well defined.
			\end{proof}
			
			This rational fibration cannot extend to an actual morphism $\calM_S \to \matP^3$ with fibre $\mathsf{res}(\calM_X)$ over $[X]\in\matP^3$ directly. Indeed, the image of $\calM_X$ in $\calM_S$ is singular (Theorem \ref{thmrestotal}).

			To conclude this section, we show that $\calM_S$ admits a birational model for which there is an actual Lagrangian fibration over $\matP^3$ with generic fibre $\Pic^2(\Gamma)$, which can be thought as "filling up" the rational fibration in Corollary \ref{corratlagfib}.
			
			\begin{theorem}\label{thmactuallagfibrationbiratmodel}
				The map $\Phi^!\colon \calM_{\mathfrak{X}} \to \Pic^2_\alpha(\mathfrak{G})$ induces a birational map $\calM_S \dashrightarrow \calM_S[5,2,6]$. Moreover, the functor $\phi_{10}$ induces a birational map
				$\calM_{(S',\alpha)} \dashrightarrow \calM_S[5,2,6],$
				and the Lagrangian fibration
				$$\calM_{(S',\alpha)} \to \matP^3$$
				is birational to the rational fibration $\calM_S\dashrightarrow \matP^3$ defined in Corollary~\ref{corratlagfib}.
				
			\end{theorem}

			\begin{proof}
				Note that $\Phi^!\colon \calM_\mathfrak{X}^o \xrightarrow{\sim} \Pic^2_\alpha(\mathfrak{G})^o$ is an isomorphism (once more, it follows from Proposition~\ref{PropCriterePlatitudeParFibres} and the fact that it is fibrewise an isomorphism), inducing a birational map $\calM_S\dashrightarrow \calM_S[5,2,6]$ by Proposition~\ref{PropResMorphGiveOpenImmersion}.
				
				On the other hand, for any $[X]\in\calW$ (and the corresponding curve $\Gamma$), the functor $\phi_{10}$ sends a sheaf of the form $(i_{\Gamma S'})_*L \in \calM_{(S',\alpha)}$, for $L\in\Pic^2\Gamma$, to $i^*_{SX}\phi L\in\calM_S$ (Lemma~\ref{lemisoway}). The composition
				$$\phi_{10}^{-1} \circ \mathsf{res}_\calW'^o\colon \Pic^2_\alpha(\mathfrak{G})^o \hookrightarrow \calM_{(S',\alpha)}$$
				is an open immersion, as $\mathsf{res}_\calW'^o$ is so and $\phi_{10}$ is an equivalence. We obtain a birational map
				$$\calM_S \dashrightarrow \calM_{(S',\alpha)},$$
				defined over $\calW$.
			\end{proof}

			\section{The birational models of $\calM_S$}\label{SectionStudyBiratModelsMS}

			The goal of this section is to study the different birational models of $\calM_S$, in particular $\calM_S[5,2,6]$ and $ \calM_{(S',\alpha)}$. Using Bridgeland stability conditions and results from \cite{BMMMPwallcrossing}, we get the following.

			\begin{theorem}\label{ThmKtrivialBirModels}
				The moduli spaces $\calM_S$ and $\calM_S[5,2,6]$ are not isomorphic and are related by a flop along a $\matP^2$-bundle over $S$ which extends the construction of Theorem \ref{thmactuallagfibrationbiratmodel}. They are the only two smooth $K$-trivial birational models of $\calM_S$. Moreover, $\calM_S[5,2,6] \simeq \calM_{(S',\alpha)}$.
			\end{theorem}
			
			The theorem will be proved \S \ref{SectionIdentifyingTheBirationalModels}.
			

			\subsection{Stability conditions}\label{sectionGeneralityStabcond}
			
			We will not give the actual definition of stability conditions. We will mainly focus in geometric stability conditions, and we will later prove that computations in the geometric case are sufficient to prove Theorem \ref{ThmKtrivialBirModels}.  A good survey for the general theory is \cite{MacriSchmidtLecturesBridgelandStability}.
			
			Let $S$ be a K3 surface and $H$ an ample line bundle. We denote $\Stab(S)$ the space of stability conditions on $S$. It admits a structure of complex manifold with finite dimension \cite{BridgelandStabCondTriangCat}. Set $\Lambda\coloneqq H^0(S,\matZ)\oplus \NS(S) \oplus H^4(S,\matZ)$. We denote $v\colon K_0(S) \to \Lambda$ the Mukai vector, and $\langle \_,\_ \rangle$ the bilinear form on $\Lambda$ given by $\langle (v_0,v_1,v_2),(v_0',v_1',v_2')\rangle =v_1\cdot v_1' - v_0v_2' - v_2v_0'.$
			
			Pick $\alpha,\beta \in \matR$, $\alpha>0$. For an object $F\in \D^b(S)$, set the $\beta$-slope as
			$$\mu_\beta (F)\coloneqq \dfrac{H\cdot c_1F}{H^2\rk F}-\beta.$$
			First, we consider the abelian full subcategory $\Coh^\beta(S)\coloneqq\langle \mathbf{T}^\beta, \mathbf{F}^\beta[1] \rangle$, where
			\begin{eqnarray*}
				&\mathbf{T}^\beta &\coloneqq \{F \in\Coh(S) \ \colon \text{ all quotients } F\twoheadrightarrow T \text{ satisfy } \mu_\beta(T)>0 \} \\
				&\mathbf{F}^\beta &\coloneqq \{F \in\Coh(S) \ \colon \text{ all subobjects } E\hookrightarrow F \text{ satisfy } \mu_\beta(E)\leq 0 \}.
			\end{eqnarray*}
			Any object $F\in\Coh^\beta(S)$ is such that $\calH^i(F)=0$ for $i\neq 0,-1$, $\calH^0(F)\in \mathbf{T}^\beta$ and $\calH^{-1}(F)\in \mathbf{F}^\beta$. We consider the stability function $Z\colon \Lambda \to \matC$ given by
			\begin{eqnarray*}
				Z_{\alpha,\beta}(v_0,v_1,v_2) &=& \left(e^{i\alpha H + \beta H},(v_0,v_1,v_2)\right) \\
				&=&i\alpha H\cdot (v_1-\beta v_0H) -v_2 + \beta H\cdot v_1 + \frac{H^2}{2}(\alpha^2-\beta^2)v_0.
			\end{eqnarray*}
			For a Mukai vector $v\in\Lambda$, we consider the slope-function
			$$\mu_Z (v)\coloneqq \dfrac{-\Re Z(v)}{\Im Z(v)}.$$
			When $F\in\D^b(S)$, we will use the abuse of notation $\mu_Z(F)\coloneqq \mu_Z(Z(v(F)))$.
			
			\begin{definition}
				An object $F\in\D^b(S)$ is called $\sigma_{\alpha,\beta}$-\textit{stable}, resp. $\sigma_{\alpha,\beta}$-\textit{semistable}, if there is some $k\in\matZ$ such that $F[k]\in\Coh^\beta(S)$, and the object $F[k]$ is slope-stable, resp. slope-semistable, with respect to the slope function $\mu_Z$ in the category $\Coh^\beta(S)$.
			\end{definition}
			
			More precisely, $F$ is $\sigma_{\alpha,\beta}$-(semi)stable if there exists $k\in\matZ$ such that $F[k]\in\Coh^\beta(S)$ and for any subobject $E\hookrightarrow F[k]$ in $\Coh^\beta(S)$ with quotient $T$, we have
			$$\mu_Z(E) < (\leq) \mu_Z(T).$$

			\begin{proposition}[\cite{BridgelandStabCondK3}]\label{PropZalphabetaIsStabCondition}
				The pair $\sigma_{\alpha,\beta}\coloneqq\left(\Coh^\beta(S), Z_{\alpha,\beta}\right)$ defines a stability conditions on $\D^b(S)$ if $\Re Z_{\alpha,\beta}(\delta)>0$ for all roots $\delta\in \Lambda$, $\delta^2=-2$ with $\rk(\delta)>0$ and $\mu_{\beta}(\delta)=0$. In particular, it holds for $\alpha^2H^2\geq 2$.
				
			\end{proposition}

			In \cite[\S10, \S 11]{BridgelandStabCondK3}  Bridgeland proves that any stability condition such that all skyscraper sheaves $k(x)$ are stable is of the form $\sigma_{\alpha,\beta}$ for some $\alpha,\beta$ up to the right-action of $\widetilde{\GL}_2(\matR)$ (the universal cover of the space $\GL_2^+(\matR)$ of matrices with positive determinant). This action composes $Z$ with the matrix seen as a linear endomorphism of $\matC\simeq \matR^2$ and relabels the phases $\phi$. We call such a stability condition \textit{geometric}. We denote by $U(S)\subset \Stab(S)$ the open subset of geometric stability conditions and $\Stab^+(S)$ as the connected component of $\Stab(S)$ containing all geometric stability conditions.

			One could ask how does the spaces of semistable objects vary when varying the stability condition. An answer is given by the following well-known theorem.
			
			\begin{theorem}[\cite{BMSpaceOfStabCondOnLocalProjectivePlane}, Prop. 3.3]\label{ThmWallAndChamberDecompStabGen}
				Let $v\in\Lambda$ be a fixed primitive class. Then $\Stab(S)$ admits a \textit{wall and chamber} decomposition (depending on $v$), that is there exists a locally finite family of real codimension $1$ submanifolds with boundaries, called \textit{walls}, with the following properties.
				\begin{enumerate}
					\item For two stability conditions $\sigma_1,\sigma_2$ lying in the same \textit{chamber} (that is, connected component of the complement of the union of all walls), an object $E$ is $\sigma_1$-(semi)stable if and only if it is $\sigma_2$-(semi)stable.
					\item A stability condition $\sigma$ lies on a wall if and only if there exists a strictly $\sigma$-semistable object.
				\end{enumerate}
				
			\end{theorem}
			
			We call a stability condition \textit{generic} if it does not lie on a wall. When fixing a Mukai vector $v\in\Lambda$, we would like to define $\calM_\sigma[v]$, resp. $\calM_\sigma[v]^{st}$ as the coarse moduli spaces of ($S$-equivalent classes of) semistable, resp. stable objects in $\D^b(S)$ of class $v$.
			
			\begin{theorem}[\cite{BMMMPwallcrossing}, Thm. 2.15]\label{ThmModuliStableObjectIsHKvaranddim}
				Let $v\in\Lambda$ be a primitive Mukai vector, and $\sigma\in\Stab^+(S)$ be a generic stability condition on $S$. Then $\calM_\sigma[v]$ exists as a smooth projective irreducible holomorphic symplectic variety. Moreover, either $\calM_\sigma[v]$ is empty or $v^2\geq -2$ and $\dim\calM_\sigma[v]=v^2+2$ and $\calM_\sigma[v]^{st}\neq \emptyset$.
			\end{theorem}
			
			Finally, we will use the existence of a Gieseker chamber, as explained below.
			
			\begin{theorem}[\cite{BridgelandStabCondK3}, \S 14.2]\label{ThmGiesekerChamberModuliStabCond}
				Let $v=(v_0,v_1,v_2)$ be a primitive class with $v_0>0$. Then there exists $\alpha_0>0$ such that for any $\alpha\geq \alpha_0$ and all $\beta < \frac{Hv_1}{H^2v_0}$, an object $F\in\D^b(S)$ of class $v(F)=v$ is $\sigma_{\alpha,\beta}$-stable if and only if it is the shift of a Gieseker-stable sheaf. In particular, the moduli space $\calM_{\sigma_{\alpha,\beta}}[v]$ is isomorphic to the moduli space $\calM_S[v]$ of Gieseker-stable sheaves on $S$ of class $v$.
			\end{theorem}

			\subsection{From stability conditions to $\NS(\calM_S)$}
			
			Let us quickly recall some facts on hyperk\"ahler manifolds. We refer to Debarre's survey \cite{DebarreHKManifolds} for the statements.
			
			Let $M$ be a projective hyperk\"ahler manifold. The cohomology space $H^2(M,\matZ)$ can be equipped with a natural quadratic form $q\colon H^2(M,\matZ)\to\matZ$, called \textit{Beauville-Bogomolov form}, and $q$ induces a bilinear form on $\NS(M)_\matR$.
			
			Let $\Pos(M)$ be the \textit{strictly positive cone} of $M$, i.e. the connected component of $\{ D\in\NS(M)_\matR \ | \ q(D)>0 \}$ containing the ample classes. The \textit{movable cone} $\Mov(M)\subset \NS(M)_\matR$ of $M$ is define as the cone generated by classes of divisors $D$ such that the base locus of $|D|$ has codimension at least $2$.  
			Denoting $\Nef(M)$ the cone generated by Nef divisors on $M$ (it identifies with the closure $\overline{\Amp}(M)$ of the ample cone), we have the inclusions
			$$\Nef(M)\subset \overline{\Mov}(M) \subset \overline{\Pos}(M).$$
			From \cite{HassetTschinkelMovingAmpleConesHKFourfolds}, $\overline{\Mov}(M)$ admits a wall-and-chamber decompositions, each chamber corresponding the the image of $\Nef(M')$ for $M'$ a birational model of $M$, and moreover the nef cone of each birational model of $M$ appears as a chamber. This wall and decomposition can actually be seen in $\Stab^+(S)$ thanks to the work of Bayer and Macr\`i.
			
			\begin{theorem}[\cite{BMMMPwallcrossing}, Thm. $1.1$ and $1.2$]\label{ThmBMChambStabisChambMov}
				For $\sigma, \tau\in\Stab^+(S)$ generic, we have $\calM_\sigma[v]\simeq \calM_\tau[v]$.
				Fix a base point $\sigma\in\Stab^+(S)$. There is a map
				\begin{eqnarray}
					l\colon \Stab^+(\calM_\sigma[v]) \to \Mov(\calM_\sigma[v])
				\end{eqnarray}
				such that the following holds.
				
				\begin{enumerate}
					
					\item The image of $l$ is the cone of big movable divisor $\Mov(\calM_\sigma[v])\cap \Pos(\calM_\sigma[v])$.
					
					\item For any generic stability condition $\tau\in\Stab^+(S)$, the image $l(\tau)$ lies in the chamber of $\Mov(\calM_\sigma[v])$ which correspond to the birational model $\calM_\tau[v]$ of $\calM_\sigma[v]$. In particular, all smooth $K$-trivial birational model of $\calM_\sigma[v]$ appears as $\calM_\calC[v]$ for some chamber $\calC\subset \Stab^+(S)$.
					
					\item For any chamber $\calC\subset \Stab^+(S)$, we have $l(\calC)=\Amp(\calM_\calC[v])$.
					
				\end{enumerate}
			\end{theorem}

			Crossing a wall in $\Stab^+(S)$ produces a birational transformation between the moduli spaces of the adjacent chambers. The type of birational transformation is described in \cite[Thm. 5.7]{BMMMPwallcrossing}, and can be studied in a lattice-theoretic way.

			\subsection{Computing the walls}\label{SectionComputationsWalls}

			Consider the K3 surface we studied \S \ref{sectionLagrangianFibration}. Recall that $S$ is a K3 surface of genus $9$, with $\Pic(S)=\matZ \langle H \rangle$ where $H$ is an ample divisor of square $H^2=16$.  We fix $v=(2,1,3)$ (which correspond to sheaves with $c_2=7$, see \S \ref{SectionCohomModuliSheaves}). Our goal is to study 
			$$\calM_S\coloneqq \calM_S(2,1,7)=\calM_S[2,1,3] = \calM_S[v].$$

			First, we show that it is sufficient to consider stability conditions of the form $\sigma_{\alpha,\beta}$ as constructed \S \ref{sectionGeneralityStabcond}.

			\begin{proposition}\label{propallstabaregeometric}
				Let $\sigma\in\Stab^+(S)$ be a generic stability condition. Then there is an autoequivalence $\phi\in\Aut(\D^b(S))$ with $\phi^H(v)=v$ such that $\phi(\sigma)$ lies in $U(S)$. Moreover, the moduli space $\calM_\sigma[v]$ is isomorphic to $\calM_{\sigma_{\alpha,\beta}}[v]$ for some $\alpha,\beta$.
			\end{proposition}

			\begin{proof}
				From the proof of \cite[Prop. 13.2]{BridgelandStabCondK3}  there exists an autoequivalence $\phi\in\Aut(\D^b(S))$ such that $\phi(\sigma)\in\overline{U(S)}$, which is the composite of autoequivalences either of the form $T_A^2$, the square of a spherical twist along a spherical vector bundle $A$ on $S$, or of the form $T_{\calO_C(k)}$, the spherical twist along the structure sheaf of a nonsingular rational curve $C\subset S$. But the latter case cannot occur as $\Pic(S)=\matZ H$, $H^2=16$, and any smooth rational curve satisfies $C^2=-2$. Now the fact $\phi(v)=v$ follows from the remark that $T_A^2$ acts trivially in cohomology ($T_A$ acts by reflection in the hyperplane orthogonal to $v(A)$).
				Moreover, as we assumed that $\sigma$ does not lie on a wall, $\phi(\sigma)$ cannot be on the boundary $\partial U(S)$ as the latter is covered by walls.
				Finally, note that the action of $\widetilde{GL}_2(\matR)$ does not affect the moduli space of stable objects as it only changes the phases. By definition of $U(S)$, some element of this group must send $\phi\sigma$ to a stability condition of the form $\sigma_{\alpha,\beta}$ as constructed above.
				
				To sum up, there is an autoequivalence $\phi\in\Aut(\D^b(S))$ acting trivially on cohomology and an element $m\in\widetilde{GL}_2(\matR)$ which give
				$$\calM_\sigma[v] \xrightarrow{\sim} \calM_{\phi \sigma}[v] \xrightarrow{\sim} \calM_{m\phi\sigma}[v] \simeq \calM_{\sigma_{\alpha,\beta}}[v].$$
			\end{proof}

			Denote $\matH=\{(\beta,\alpha)\in\matR^2 \ | \ \alpha>0\}$ the open upper halfplane in $\matR^2$. In view of Proposition \ref{propallstabaregeometric}, we must compute the walls lying in $\matH$. To do so, we will use the very useful computations made by Maciocia in \cite{MaciociaComputationsWalls}. 
			We write $Z_{\alpha,\beta}$ for the central charge of a stability condition of the form $\sigma_{\alpha,\beta}$. We drop $(\alpha,\beta)$ and simply write $Z\coloneqq Z_{\alpha,\beta}$ when the context is clear.

			\begin{definition}
				Given a class $0\neq w\in\Lambda$, we define the \textit{numerical wall generated by} $w$ as the nonempty subset of $\Stab(S)$ given by
				\begin{eqnarray*}
					W(w) = \{ \sigma=(\calP,Z) \in \Stab(S) | \ \Re Z(v)\cdot\Im Z(w) = \Re Z(w)\cdot\Im Z(v)\}
				\end{eqnarray*}
			\end{definition}
			
			In particular, any \textit{actual} wall of Theorem \ref{ThmWallAndChamberDecompStabGen} for which there exists an inclusion $E_w \hookrightarrow F$  with $v(E_w)=w$ lies in the numerical wall $W(w)$. We say that a point $\sigma\in W(w)$ is an \textit{actual point} if it lies in an actual wall.
			
			Let us focus on the case of stability condition of the form $\sigma_{\alpha,\beta}$. If a numerical wall $W(w)\subset\matH$ is actual at a point, then it remains actual on the connected component of $W(w)$, in other words an actual wall is a subset of a numerical wall cut out by holes (correponding to the existence of spherical objects, as in Proposition \ref{PropZalphabetaIsStabCondition}) or by $\{\alpha=0 \}$, see \cite[\S 6.4]{MacriSchmidtLecturesBridgelandStability}.

			\begin{proposition}[\cite{MaciociaComputationsWalls}]\label{PropMaciociaComputations}
				Let $w=(w_0,w_1,w_2)\in\Lambda$, $w\neq0$, and consider the associated numerical wall $W(w)\subset \matH$. Assume $W(w)$ is not $0$-dimensional. Let $(\beta,\alpha)\in W(w)$.
				
				\begin{enumerate}
					\item Either $\beta=\dfrac{v_1}{v_0}$, $\alpha >0$, or $\beta\neq\dfrac{v_1}{v_0}$ and $(\beta,\alpha)$ lies in a semicircle of center $(C,0)$ and radius $R$, where
					$$ C = \frac{v_0w_2-v_2w_0}{H^2(v_0w_1-v_1w_0)} \ \text{ and } \ R = \sqrt{\left(C-\frac{v_1}{v_0}\right)^2-Q},$$
					with $Q=\dfrac{\langle v,v \rangle}{H^2v_0^2}$.
					
					In the first case, we call $W_v(w)\coloneqq\{\beta = \frac{v_1}{v_0}\}$ the \textbf{vertical wall}, the semicircle in the second case is called a \textbf{semicircular wall}.
					
					\item Assume $Q\geq 0$. If $W(w)$ is a semicircular wall, then the center $C$ satisfies either
					\begin{eqnarray*}
						C < \frac{v_1}{v_0}-\sqrt{Q} \text{ \ or \ } \frac{v_1}{v_0}+\sqrt{Q}<C.
					\end{eqnarray*}
					
					Moreover, $W(w)$ must intersect either the ray $\{\beta = \dfrac{v_1}{v_0}-\sqrt{Q}\}$ or the ray $\{\beta = \dfrac{v_1}{v_0}+\sqrt{Q}\}$ depending on its position relative to the vertical wall.
				\end{enumerate}
			\end{proposition}

			\begin{remark}\label{RmkBrokenWallsDoNotMatter}
				Note that the only walls we are interested in are actual walls $W$ remaining actual along the whole numerical wall $W_{num}$, except on the holes in $W_{num}$ arising from spherical classes (see Proposition \ref{PropZalphabetaIsStabCondition}). Indeed, assume $W_{num}=W(w)$ for some class $w\in \widetilde{H}(S,\matZ)$, and let $\sigma\coloneqq\sigma_{\alpha,\beta}\in W(w)\mysetminus W$ be a stability condition (hence not on a hole). Then $\sigma$ lies in some chamber $\calC\subset \Stab^+(S)$, and its image $l(\sigma)$ lies in the open ample cone $\Amp(\calM_\calC[v])$. Since $\Amp(\calM_\calC[v])$ is a cone, the whole ray $\matR_{>0}\cdot  l(\sigma)$ lies in $\Amp(\calM_\calC[v])$. But this halfline contains $W(w)$ (this is a consequence of the definition of $l$, see \cite{BMMMPwallcrossing}, Theorem 10.2), in particular for any point $\sigma_0\in W$, given two stability condition $\sigma_\pm$ near $\sigma_0$ in each adjacent chamber, the corresponding image $l(\sigma_\pm)$ both lie in $\Amp(\calM_\calC[v])$, and hence $\calM_{\sigma_+}[v] = \calM_{\sigma_-}[v]=\calM_\calC[v]$.
			\end{remark}

			Therefore, we can compute the wall and chamber decomposition with respect to $v=(2,1,3)$ thanks to the description of Proposition \ref{PropMaciociaComputations}. Let $\alpha,\beta\in\matH$. Assume $\widetilde{F}\in\calM_{\sigma_{\alpha,\beta}}[2,1,3]$. By definition, there is an integer $k\in\matZ$ with $\widetilde{F}[k]\in\Coh^\beta(S)$. If $k$ is even, we have $v(\widetilde{F}[k])=(2,1,3)$, and if $k$ is odd we have $v(\widetilde{F}[k])=(-2,-1,-3)$.

			\subsubsection{The vertical wall}

			By Proposition~\ref{PropMaciociaComputations}, there is (at least numerically) a vertical wall 
			$$W_v=\{\beta=1/2\}$$
			in $\matH$, given by the class $(2,1,4)\in\Lambda$. In fact, this wall is an actual wall. Indeed, following Proposition \ref{PropDescriptionSheafMx217},  pick $F\in\calM_S[2,1,3]$ not locally free, the sheaf $E\coloneqq F^{**}$ is a stable vector bundle such that $E\in\calM_S[2,1,4]$. We have the exact sequence
			$$0 \to F \to E \to \calO_x\to 0$$
			for some point $x\in S$. This induces the exact triangle
			\begin{eqnarray}\label{ExactTriangleForVerticalWall}
				\calO_x \to F[1] \to E[1].
			\end{eqnarray}
			We claim that (\ref{ExactTriangleForVerticalWall}) induces an exact sequence in $\Coh^{1/2}(S)$. Indeed $E[1],F[1]\in \mathbf{F}^{1/2}[1]$ and $\calO_x\in \mathbf{T}^{1/2}$.

			\subsubsection{Left side of the halfplane}\label{subsubsectionleftside}

			Consider a wall left to the vertical wall $W_v$, given by an exact sequence
			\begin{eqnarray}\label{EqnExSeqEFTforwalls}
				E \hookrightarrow F \rightarrow T
			\end{eqnarray}
			in $\Coh^\beta(S)$, with $F\coloneqq\widetilde{F}[k]\in\Coh^\beta(S)$, $v(\widetilde{F})=(2,1,3)$. From $\beta<1/2$ we must have $k$ even, that is $v(F)=(2,1,3)$. Denote $v(E)=w=(w_0,w_1,w_2)$ and $v(T)=t=(t_0,t_1,t_2)$. Note that $E$ and $T$ will play a symmetric role in the following, hence we can assume that $w_0>0$ for $2=v_0=w_0+t_0$.
			
			Using Proposition \ref{PropMaciociaComputations} in this case, the wall is a semicircular wall with center $C$ and radius $R$. We get
			\begin{eqnarray}
				C < \frac{1}{4} \label{twoboundsC}
			\end{eqnarray}
			Moreover, any semicircular wall must intersect the ray $\{\beta=\frac{1}{4}\}$. Note that for $\beta=\frac{1}{4}$, there is no class $\delta=(\delta_0,\delta_1,\delta_2)\in\Lambda$ with $\delta_0>0$, $\delta^2=-2$ and $\mu_{1/4}(\delta)=0$. Indeed, these conditions on $\delta$ give
			\begin{eqnarray*}
				8\delta_1^2 &=& \delta_0\delta_2-1 \\
				\delta_0 &=& 4\delta_1,
			\end{eqnarray*}
			which is impossible because $\delta_0,\delta_1,\delta_2\in\matZ$. From Proposition \ref{PropZalphabetaIsStabCondition}, $Z\coloneqq Z_{\alpha,\frac{1}{4}}$  defines a stability condition. 
			
			In view of Remark \ref{RmkBrokenWallsDoNotMatter}, we can assume that the exact sequence (\ref{EqnExSeqEFTforwalls}) hold at $\beta=\frac{1}{4}$. Recall that the imaginary and real part of $Z(-)$ are additive on exact sequence in $\Coh^\beta(S)$, and since $\mu_Z(F)<\infty$ we get $0<\Im(Zw)<\Im(Zv)$ and  $0<\Im(Zt)<\Im(Zv)$. For $\beta=\frac{1}{4}$, it gives
			\begin{eqnarray*}
				& & 0 < w_1-\frac{w_0}{4} < \frac{1}{2} \\
				&\iff & \frac{w_0}{4} < w_1 < \frac{w_0}{4}+\frac{1}{2}
			\end{eqnarray*}
			Since $w_1$ and $w_0$ are integers, we obtain $w_0=4n+3$, $w_1=n+1$ for some $n\in\matZ$. Moreover we assumed $w_0>0$, so $n\geq0$ and in particular $2w_1-w_0=-2n-1<0$.
			The inequality $C<\frac{1}{4}$ gives
			\begin{eqnarray}
				& & \frac{2w_2-3w_0}{16(2w_1-w_0)} < \frac{1}{4} \nonumber \\
				& \iff & 2w_2-3w_0 > 4(2w_1-w_0) \nonumber \\
				& \iff & w_2 > 4w_1 - \frac{1}{2}w_0 = 2n+\frac{5}{2} \label{infboundw2}
			\end{eqnarray}
			
			To obtain more bounds, we need to study $E$ with more details.
			
			\begin{lemma}\label{lemmaEspherical}
				The object $E$ is a $\sigma_{\alpha,1/4}$-\textit{stable} object, in particular it satisfies
				\begin{eqnarray}\label{Eissphericalequality}
					v(E)^2=16w_1^2-2w_0w_2 \geq -2.
				\end{eqnarray}
			\end{lemma}
			
			\begin{proof}
				If $E$ is stable, then $\Hom_{\Coh^\beta(S)}(E,E)=\Hom_{\D^b(S)}(E,E)=\matC$ because any stable object is simple. Hence by Serre duality we get
				$v(E)^2=-\chi(E,E)=-(2-\Ext^1(E,E))\geq -2.$
				Moreover, note that for $\beta=\frac{1}{4}$, any semistable object of class $a=(a_0,a_1,a_2)$ satisfies 
				\begin{eqnarray}\label{EqnMinimalImZcannotbeDestab}
					\Im Z(a)=\alpha H^2\frac{1}{4}(4a_1-a_0) \geq \frac{1}{4}\alpha H^2.
				\end{eqnarray}
				But we have $\Im Z(v(E))=\alpha H^2(w_1-\beta w_0)=\frac{1}{4}\alpha H^2$, so $E$ cannot be strictly semistable, otherwise any of its proper Jordan-H\"older factor $A$ would satisfy $\Im Z(v(A))< \frac{1}{4}\alpha H^2$ which contradicts (\ref{EqnMinimalImZcannotbeDestab}).
			\end{proof}

			Using (\ref{Eissphericalequality}) we get 
			\begin{eqnarray*}
				& & 16w_1^2-2w_0w_2\geq -2 \\
				&\iff & w_2 \leq 8\frac{w_1^2}{w_0} +\frac{1}{w_0}
			\end{eqnarray*}
			By straightforward computations, we have
			$$8\frac{w_1^2}{w_0} = 8\frac{(n+1)^2}{4n+3} = 2n+\frac{5}{2}+\frac{1}{2(4n+3)}.$$
			Combined with (\ref{infboundw2}) we get
			$$2n+\frac{5}{2} < w_2 \leq 2n+\frac{5}{2} + \frac{3}{2(4n+3)}.$$
			Hence the only possibility is $n=0$, which gives $w_0=3$, $w_1=1$ and $w_2=3$. We get $C=\frac{3}{16}$.

			\begin{proposition}\label{PropCircularWallDescriptionLeftSide}
				The circular numerical wall $W_l(3,1,3)$ with center $(C=\frac{3}{16},0)$ and radius $R=\frac{3}{16}$ is an actual wall $W_l$. It is induced by the following exact sequences 
				\begin{eqnarray}
					\calU_S^\vee \hookrightarrow F_S \twoheadrightarrow \calI_x^\vee[1] &\text{ if } & \beta<\frac{1}{3} \label{EqnDestabExSeqCircularWall} \\
					(\phi\phi^!F)_S \hookrightarrow F_S \twoheadrightarrow \calU^\vee_S[1] & \text{ if } & \frac{1}{3}< \beta \label{EqnDestabExSeqCircularWall2}
				\end{eqnarray}
				in $\Coh^{\beta}(S)$, where $F\in\calM_X$ is a non globally generated sheaf, for $[X] \in \calW$ a Fano threefold containing $S$ as in \S (\ref{setupFanog9}), and $\calI_x^\vee$ is the derived dual of the ideal sheaf of the point $x\in S$.
			\end{proposition}
			
			For the definition of $\phi\phi^!$, see \S \ref{sectionHPD}. When $F\in\calM_X$ is globally generated, then (\ref{EqnDestabExSeqCircularWall}) cannot happen by Proposition~\ref{PropDescriptionSheafMx217}. On the otherhand, given another $G\in\calM_X$ we know that $(\phi\phi^!F)_S\simeq (\phi\phi^!G)_S$ implies that $F\simeq G$ (Proposition~\ref{PropResMorphGiveOpenImmersion}), in particular $F_S$ and $G_S$ do not become $S$-equivalent on the wall.
			
			
			\begin{proof}
				Recall from Proposition \ref{PropDescriptionSheafMx217} that for any non globally generated sheaf $F\in\calM_X$, there is a map $f\colon \calU^\vee_S \to F_S$ which induces an exact sequence
				\begin{eqnarray}\label{EqnDestabLongExSeqSheavesCircularWall}
					0 \to \calO_S  \to \calU^\vee_S \xrightarrow{f} F_S \to \calO_x \to 0.
				\end{eqnarray}
				Consider the cone $C(f)\in\D^b(S)$. It lies in $\Coh^\beta(S)$ for any $\beta>0$. A similar argument as in Lemma~\ref{lemmaEspherical} shows that $C(f)$ is $\mu_Z$-stable for any $0<\beta<1$.

				\begin{lemma}\label{LemmaM(-100)}
					We have $\calM_{\sigma_{\alpha,\beta}}[-1,0,0] = \{ \calI_x^\vee[1] \ | \ x\in S\} \simeq S$ for any $\beta>0$.
				\end{lemma}

				\begin{proof}
					Once again, it is easy to prove that for any $0<\beta<1$ and $x\in S$, the object $\calI_x^\vee[1]$ is $\mu_Z$-stable. By Theorem \ref{ThmModuliStableObjectIsHKvaranddim}, $\calM_\sigma[-1,0,0]$ is a K3 surface, and since it contains $S$ it must be isomorphic to $S$.
				\end{proof}

				Hence $C(f) = \calI_x^\vee[1]$ for the point $x\in S$ appearing in (\ref{EqnDestabLongExSeqSheavesCircularWall}). We obtain the desired exact sequence (\ref{EqnDestabExSeqCircularWall}) in $\Coh^\beta(S)$ for all $0<\beta<\frac{1}{3}$.
				At the point $\sigma_0=(\frac{1}{3},\alpha)\in W_l(3,1,3)$, $\calU^\vee_S$ is a spherical object with $Z(v(\calU_S^\vee))=0$, so $\sigma_0$ is \textit{not} a stability condition (in other words, $\sigma_0$ is a hole of $W_l$).
				Finally, for $\frac{1}{3}<\beta$ the extension (\ref{UFSexseq}) give the exact sequence (\ref{EqnDestabExSeqCircularWall2}) in $\Coh^\beta(S)$, and it is easily seen to remain valid all along this part of the numerical wall $W_l(3,1,3)$.
				
			\end{proof}

			\subsubsection{Right side of the halfplane}

			In a analogous way as in \S \ref{subsubsectionleftside}, we can prove that there is no hole on the ray $\{\beta = 3/4\}$, so we restrict ourselves to $\beta=\frac{3}{4}$. Similar computations lead to two possible numerical walls, induced respectively by $w=(1,1,8)$ giving $C_8=\frac{13}{16}$ and $R_8=\frac{3}{16}$, and $w=(1,1,9)$ giving $C_9=\frac{15}{16}$ and $R_9=\frac{\sqrt{33}}{16}$.

			\begin{proposition}
				The numerical wall of center $(C_9=\frac{15}{16},0)$ and radius $R_9=\frac{\sqrt{33}}{16}$ is not an actual wall.
			\end{proposition}
			
			\begin{proof}
				
				Note that the circle of center $C_9$ and radius $R_9$ cross the ray $\{\beta=2/3\}$ at $\alpha^2=R_9^2 - (\frac{2}{3}-C)^2=\frac{1}{18}$.
				
				\begin{lemma}
					Any vector $\delta=(\delta_0,\delta_1,\delta_2)\in\Lambda$ with $\delta_0>0$, $\delta^2=-2$ and $\mu_{2/3}(\delta)=0$ satisfies $\Re Z_{\alpha,2/3}(\delta)>0$ whenever $\alpha^2>1/72$.
				\end{lemma}
				
				\begin{proof}
					Rewriting the equations on $\delta$ give
					\begin{eqnarray*}
						8\delta_1^2 &=& \delta_0\delta_2-1 \\
						2\delta_0 &=& 3\delta_1.
					\end{eqnarray*}
					In particular, we get $\delta_1,\delta_2>0$ and $\dfrac{\delta_2}{\delta_1}=\dfrac{16}{3}+\dfrac{2}{3\delta_1^2}$. We get
					\begin{eqnarray*}
						\Re Z(\delta)>0 &\iff & \delta_1(\frac{16}{3}+12\alpha^2) > \delta_2 \\
						&\iff & \frac{16}{3}+12\alpha^2 > \frac{\delta_2}{\delta_1}= \frac{16}{3}+\frac{2}{3\delta_1^2} \\
						&\iff & \alpha^2 > \frac{1}{18\delta_1^2}
					\end{eqnarray*}
					Since $\delta_0=\frac{3}{2}\delta_1$, we have $\delta_1\geq 2$. In particular $\alpha^2>1/72$ works for all $\delta$'s.

				\end{proof}

				Let $(\frac{3}{4},\alpha_0)$ be the intersection of the numerical wall $W(1,1,9)$ with the ray $\{\beta=\frac{3}{4}\}$, and consider a destabilizing exact sequence
				\begin{eqnarray}\label{EqnEFTExSeqCohBetaRightSide}
					E \hookrightarrow F \twoheadrightarrow T.
				\end{eqnarray}
				Recall that $v(F)=(-2,-1,-3)$, and we assume by symmetry that $v(E)=(1,1,9)$ and $v(T)=(-3,-2,-12)$.
				We can assume that $F$ is $\sigma_{\alpha_0+\varepsilon,\frac{3}{4}}$-stable, and since $W(1,1,9)$ is the largest circular wall crossing $\{\beta=\frac{3}{4}\}$ (on the right side of the vertical wall $W_v$), we can assume that $F$ is $\sigma_{\alpha,\frac{3}{4}}$-stable for all $\alpha\gg 0$ (in fact, it is $\sigma$-stable for all $\sigma$ above the wall). By a similar argument as for Theorem \ref{ThmGiesekerChamberModuliStabCond} (see \cite[Lem. 6.18]{MacriSchmidtLecturesBridgelandStability}), one can prove that $\calH^0(F)$ is a torsion sheaf supported in dimension $0$ and $\calH^{-1}F$ is a slope-semistable torsion-free sheaf. 
				
				Consider the long exact sequence (in $\Coh(S)$)
				$$0 \to  \calH^{-1}E \to \calH^{-1}F \to \calH^{-1}T \to \calH^0E \to \calH^0F \to \calH^0T \to 0$$
				induced by (\ref{EqnEFTExSeqCohBetaRightSide}). We have $\rk(\calH^0T)=0=\deg(\calH^0T)$, hence $\rk(\calH^{-1}T)=3$ and $\deg(\calH^{-1}T)=2$. In particular, any subsheaf $G \subset \calH^{-1}T$ satisfying $\mu_H(G)<\frac{3}{4}$ also satisfies $\mu_H(G)<\frac{2}{3}$, thus $\calH^{-1}T\in \mathbf{F}^{2/3}$. Moreover, either $\calH^{-1}E=0$ (in which case $E=\calH^0E$ has rank $1$ and hence is $\mu_H$-stable) or $\mu_H(\calH^{-1}E) \leq \mu_H(\calH^{0}F)=\frac{1}{2}$, thus $\calH^{-1}E\in \mathbf{F}^{2/3}$. From this observation, we deduce that the exact (in $\Coh^{3/4}(S)$) sequence (\ref{EqnEFTExSeqCohBetaRightSide}) still holds in $\Coh^{2/3}(S)$. But then at $\beta=2/3$ we have $Z(T)=0$ which is absurd (this would be in contradiction with $F$ being $\sigma_{\alpha_0,\frac{2}{3}}$-stable, since $(\alpha_0,\frac{2}{3})$ is above the wall).
			\end{proof}

			We will see next section (see Proposition \ref{PropRightWallIsReflectLeftWall}) that $W_r$ is an actual wall which is the reflection of $W_l$ in the vertical wall $W_v$. In particular, $W_l$ and $W_r$ induce the same wall in $\Mov(\calM_S)$.

			\subsection{Crossing the walls}\label{SectionCrossingTheWalls}

			In this section we study the wallcrossing of  $W_v$ (vertical wall), $W_l$ (circular wall on the left side of $W_v$) and $W_r$ (circular wall on the right side of $W_v$), following \cite[\S 5]{BMMMPwallcrossing}. By wallcrossing, we mean the birational transformation relating the two moduli spaces corresponding to the chambers on each side of the wall. For our purposes, it is enough to know that when the wall is \textit{flopping} (resp. \textit{divisorial}), the moduli spaces are related by a flop (resp. are isomorphic). In addition, a wall is called \textit{totally semistable} is all stable objects (with respect to a stability condition sufficiently close to the wall) become strictly semistable with respect to a stability condition on the wall.
			
			Consider the hyperbolic lattice 
			\begin{eqnarray}\label{EqnHHyperbolicLatticeAssociatedToAWall}
				& & \calH_W \coloneqq \{ w\in \widetilde{H}(S,\matZ) \ | \ \Im \dfrac{Z(w)}{Z(v)}=0 \text{ for all } \sigma=(Z,\calP)\in W \}
			\end{eqnarray}
			associated to each wall (see \cite[Prop. 5.1]{BMMMPwallcrossing}). Thanks to \cite[Thm. 5.7]{BMMMPwallcrossing}, the type of wall is completely determined by $\calH_W$ (we only cite here the conditions that we will use):
			
			\begin{itemize}
				\item The wall $W$ is divisorial if there exists an isotropic class $w\in\calH_W$ with $(w,v)=1$ (type \textit{Hilbert-Chow}) or $(w,v)=2$ (type \textit{Li-Gieseker-Uhlenbeck}), or a spherical class $s\in\calH_W$ with $(s,v)=0$ (type \textit{Brill-Noether}).
				\item Otherwise, $W$ is flopping if there exists a spherical class $s\in\calH_W$ with $0<(s,v)\leq (v,v)/2$.
				\item In addition, $W$ is totally semistable if and only if there exists an isotropic class $w\in\calH_W$ with $(w,v)=1$, or an \textit{effective} spherical class (in the sense of \cite[Prop. 5.5]{BMMMPwallcrossing}) $s\in \calH_W$ with $(s,v)<0$.
			\end{itemize}
			
			
			In view of the proof of Proposition \ref{propallstabaregeometric}, we see that for any wall $W$ intersecting $U(S)$, a class $w\in\Lambda$ lies in $\calH_W$ if and only if $\Im \frac{Z(w)}{Z(v)}=0$ for any $\sigma\in W\cap \matH$. In otherwords, we can focus our attention to stability conditions of the form $\sigma_{\alpha,\beta}$ only.
			
			Let us denote $\calM_{\alpha,\beta}\coloneqq\calM_{\sigma_{\alpha,\beta}}[v]$.

			\subsubsection{The vertical wall $W_v$}

			Set $W=W_v$, then
			$$\calH_W= \Span_\matZ((2,1,4), (0,0,1)) = \{(2a,a,b)\in \Lambda \ | \ a,b\in\matZ\}.$$
			In particular, for any $w\in\calH_W$, we have $w^2=4a(4a-b)$ so $w$ cannot be spherical, and $\langle v,w\rangle$ is even. Moreover, the class $(2,1,4)$ is an isotropic class lying in $\calH_W$. Therefore the wall $W$ is divisorial of type Li-Gieseker-Uhlenbeck. In particular, the wall $W$ is a \textit{bouncing wall}, in the sense that the image of the chambers on both side of the wall in $\Stab^+(S)$ are sent to the same chamber in $\Mov(\calM_S)$ \cite[Lem. 10.1]{BMMMPwallcrossing}. For $\beta_-<\frac{1}{2}<\beta_+$ close enough to $\frac{1}{2}$ and admissible $\alpha$'s, both moduli spaces $\calM_{\alpha,\beta_\pm}$ are isomorphic. 
			
			Note that by Theorem \ref{ThmGiesekerChamberModuliStabCond}, $\calM_{\alpha,\beta_-}\simeq \calM_S$, and the birational transformation when hitting the wall is the map to the Uhlenbeck compactification as constructed in \cite{LiAGInterpretationDonaldsonPolynInv}. It contracts the non-locally free sheaves $F\in\calM_S$, as we see in (\ref{ExactTriangleForVerticalWall}).

			\subsubsection{The circular wall $W_l$}

			Set $W=W_l$, then
			$$\calH_W= \Span_\matZ((0,1,3),(1,0,0)) = \{(a,b,3b)\in \Lambda \ | \ a,b\in\matZ\}.$$
			Note that there is a hole in $W_l$ at $\beta=\frac{1}{3}$. For $\beta<\frac{1}{3}$, $\calU^\vee_S$ is a stable object (see Lemma \ref{lemmaEspherical}), so this side of $W_v$ is not totally semistable. But for $\beta>\frac{1}{3}$, $-w=-(3,1,3)$ is an effective spherical class with $\langle -(3,1,3), v\rangle = -1<0$, so this portion of $\calW_l$ is totally semistable. Note that $-(3,1,3)$ corresponds to the $\sigma$-semistable object $\calU^\vee[1]$.

			For any $w=(a,b,3b)\in\calH_W$, $w^2=16b^2-6ab$ and $\langle w,v\rangle = 10b-3a$. Condition $\langle w,v\rangle=1$ gives $a=10k+3$, $b=3k+1$ with $k\in\matZ$. For such $a,b$, we cannot have $w^2=0$. In particular, there is no isotropic class $w$ satisfying $\langle w,v\rangle =1$. The same argument show that there is no isotropic class $w$ with $\langle w,v\rangle =2$ nor spherical class $s$ with $\langle s,v\rangle =0$.

			On the otherhand, $v(\calU^\vee_S)=(3,1,3)$ is a spherical class such that $0<\langle (3,1,3),v\rangle =1\leq 2=\frac{v^2}{2}$. Hence $W_l$ is a flopping wall. Pick $\sigma_\pm$ on each side of the wall near a stability condition $\sigma_{\alpha,\beta}\in W_l$ with $\beta<\frac{1}{3}$. Recall that on the wall $W_l$, any stable object of class $(-1,0,0)$ is of the for $\calI_x^\vee[1]$ for some $x\in S$. Similarly, by Theorem \ref{ThmModuliStableObjectIsHKvaranddim} the only stable object of class $(3,1,3)$ on $W_l$ is $\calU^\vee_S$. Indeed, pick a stability condition $\sigma_{\alpha,\beta}$ on $W_l$ and a stable object $E$, if $\sigma_{\alpha,\beta}$ is not generic with respect to $(3,1,3)$ then take a nearby generic stability condition $\sigma_{\alpha,\beta+\varepsilon}$. By openness of stability, $E$ is still $\sigma_{\alpha,\beta+\varepsilon}$-stable, hence isomorphic to $\calU_S^\vee$.

			\begin{proposition}\label{LemmaLeftWallIsFlopP2BundleOnS}
				The wallcrossing of $\calW_l$ is a flop along a $\matP^2$-bundle over $S$.
			\end{proposition}
			
			\begin{proof}
				We consider the part of the wall for which $\beta < 1/3$, which is simpler. 
				From Proposition \ref{PropCircularWallDescriptionLeftSide}, the objects $E\in \calM_S$ that get destabilized on the wall have the following Jordan-Hölder filtration:
				$$\calU_S^\vee \hookrightarrow E \twoheadrightarrow \calI_x^\vee[1]$$
				for some $x\in S$. We have $\ext^1(\calI_x[1]^\vee,\calU^\vee) = \hom(\calI^\vee_x,\calU^\vee) =\hom(\calU,\calI_x)=3$ (this can be proved by direct computations, using $H^0(\calU)=H^1(\calU)=0$ \cite[Lem. 3.1]{BrambillaFaenzig9}). We obtain that moving the stability condition to the wall contracts a $\matP\Ext^1(\calI_x^\vee[1],\calU^\vee_S)\simeq \matP^2$ of objects to $\calI_x^\vee[1]$ for each $x\in S$. 
				
				
				%
				
			\end{proof}

			%

			\subsubsection{The circular wall $W_r$}

			It turns out that the potential wall $W_r$ on the right side of $W_v$ is the same as the circular wall $W_l$ up to a reflection.

			\begin{proposition}\label{PropRightWallIsReflectLeftWall}
				The wall $W_r$ is the image of $W_l$ by the reflection in the vertical wall $W_v$. In particular, $W_r$ and $W_l$ have the same image in $\Mov(S)$ via the map $l\colon \Stab^+(S) \to \Mov(S)$ of Theorem \ref{ThmBMChambStabisChambMov}.
			\end{proposition}
			
			\begin{proof}
				Following \cite[Thm. 10.2]{BMMMPwallcrossing} we see that the map $l$ is the composition of a map
				$$l_0\colon \Stab^+(S) \to \Pos(\calM_S)$$
				and the action of a Weyl group generated by exceptional reflections on $\NS(\calM_S)$. Identifying $\NS(\calM_S)$ with $v^\perp$ (see \cite{YoshiokaModuliSpacesSheavesAbelianSurfaces}), computations show that a numerical wall $W(w)$ associated to a destabilizing class $w$ is sent by $l_0$ to $w^\perp\cap v^\perp$. We get
				\begin{eqnarray*}
					l_0(W_r) &=& v^\perp \cap (3,1,3)^\perp = \matR(16,13,80) \\
					l_0(W_l) &=& v^\perp\cap (1,1,8)^\perp = \matR(16,3,0).
				\end{eqnarray*}
				The reflection $\rho_D$ which identifies the chambers in both side of $W_v$ is is the reflection in the image of the vertical wall in $\NS(\calM_S)$. The latter is orthogonal to the class $D=(2,1,5)$ (in fact, this computations already appear in proof of \cite[Lem. 10.1]{BMMMPwallcrossing}). Direct computations give $\rho_D(16,13,80)=-(16,3,0)$.
			\end{proof}

			\subsubsection{Identifying the birational models}\label{SectionIdentifyingTheBirationalModels}

			Thanks to the description of the wallcrossings, we can complete the proof of Theorem \ref{ThmKtrivialBirModels} and identify the birational models of $\calM_S$ appearing in \S \ref{sectionLagrangianFibration}.

			\begin{proposition}\label{PropMS213MS526iso}
				Let $\sigma$ be a stability condition in the interior chamber cut out by $W_l$. Then 
				$$\calM_S[5,2,6] \simeq \calM_\sigma[2,1,3].$$
			\end{proposition}

			\begin{proof}
				By computations of walls, $\calM_S[5,2,6]$ is either isomorphic to $\calM_\sigma[2,1,3]$ or isomorphic to $\calM_S$. By \cite[Thm. 2.5]{YoshiokaSomeExamplesMukaiReflectionsK3Surf} (with $v=(5,2,6)$, $w=(2,1,3)$ and $v_1=(3,1,3)$ in the author's notations), we see that $\calM_S[5,2,6]$ and $\calM_S$ are not isomorphic and related by a flop. Moreover, this flop $\calM_S \dashrightarrow \calM_S[5,2,6]$ is exactly the one given by wallcrossing $W_l$, as described in Lemma \ref{LemmaLeftWallIsFlopP2BundleOnS}.
			\end{proof}

			\begin{proposition}\label{PropMS526MS010iso}
				The functor $\phi_{10}$ induces an isomorphism 
				$$\phi_{10}\colon \calM_{(S',\alpha)} \xrightarrow{\sim} \calM_S[5,2,6].$$
			\end{proposition}
			
			\begin{proof}
				The equivalence $\phi_{10}$ gives an isomorphism $\calM_{(S',\alpha)} \xrightarrow{\sim} \calM_\sigma[5,2,6]$ for some stability condition $\sigma\in \Stab(S)$. By \cite[Thm. 2.12]{BMMMPwallcrossing} we can assume $\sigma\in\Stab^+(S)$, and once again by Proposition~\ref{propallstabaregeometric} we can assume that $\sigma$ is of the form $\sigma=\sigma_{\alpha,\beta}$. In particular, since $\calM_{(S',\alpha)}$ is a smooth $K$-trivial birational model of $\calM_S$, then $\calM_\sigma[5,2,6]$ is either isomorphic to $\calM_S$ or $\calM_S[5,2,6]$. 
				
				By Theorem~\ref{thmactuallagfibrationbiratmodel}, the birational maps between these three moduli spaces are given as follows:

				$$\begin{tikzcd}[row sep=tiny]
					\calM_{(S',\alpha)} \supset \Pic^2_\alpha(\mathfrak{G})^o \ar[r,hook,"\phi_{10}"] & \calM_S[5,2,6] & \calM_\mathfrak{X}^o \subset \calM_S \ar[l,hook]  \\
					(i_{\Gamma S'})_*(\phi_{11}^!F) \ar[r,mapsto] & (\phi_{11}\phi_{11}^!F)_S & F_S \ar[l,mapsto],
				\end{tikzcd} $$
				for $F\in\calM_X$ globally generated on some $[X]\in\calW$. But $\phi_{10}$ extends to a (non-injective) regular morphism over the bigger open $\Pic^2_\alpha(\mathfrak{G})$, hence it suffices to find two elements in $\Pic^2_\alpha(\mathfrak{G})$ whose images are in the locus of $\calM_S[5,2,6]$ contracted by the flop $\calM_S[5,2,6] \dashrightarrow \calM_S$. 
				
				To do so, it is enough to find $[X]\in\calW$ and  $F, G\in\calM_X$ with $F_S\simeq G_S$ but $\phi_{11}^!F\not\simeq \phi_{11}^!G$. In the notation of Theorem~\ref{thmrestotal}, one can find $F\not\simeq G$ with $F_S\simeq G_S$, in particular $M_F=L_G$, $L_F=M_G$ but $L_F\neq M_F$. We know that $F$ and $G$ are not globally generated.  By definition, $\phi_{11}^!$ is the projection of $\D^b(X)=\langle \D^b(\Gamma),\calO_X,\calU^\vee\rangle$ to its component $\D^b(\Gamma)$, for $\Gamma$ the curve corresponding to $X$. In particular, Proposition~\ref{PropDescriptionSheafMx217} implies that $\phi_{11}^!(F)\simeq \phi_{11}^!(\calO_{L_F}(-1))$ (similarly for $G$). Now, the association $L\mapsto  \phi^!(\calO_L(-1))$, for $L$ a line in $X$, is injective \cite[Prop. 4.14]{BrambillaFaenzig9}. Hence, $L_F\neq L_G$ implies $\phi_{11}^!F\not\simeq \phi_{11}^!G$.
				
			\end{proof}

			\subsection{Movable and nef cones of $\calM_S$}\label{SectionMovNefMS}

			In this last section, we want to give a precise description of $\Pos(\calM_S)\subset \NS(\calM_S)$. Recall we have $v=(2,1,3)$. Straight computations give an orthogonal basis $\calB\coloneqq \{e_1=(0,-1,-8),e_2=(2,1,5)\}$ of $v^\perp\subset \Lambda$, that we identify with $\NS(\calM_S)$.

			\begin{proposition}\label{PropPosMovNefMS}
				The positive cone $\overline{\Pos}(\calM_S)$ is generated by $e_1+2e_2$ and $e_1-2e_2$. The closed movable cone $\overline{\Mov}(\calM_S)$ is cut out in $\overline{\Pos}(\calM_S)$ by the line $\matR e_1$, and it identifies with the chamber which contains the ample class $21e_1+8e_2$.
				
				Moreover, $\overline{\Mov}(\calM_S)$ decomposes into two chambers cut out by the line $\matR(8e_2+5e_1)$. The chamber adjacent to the line $\matR e_1$ is $\Nef(\calM_S)$ and correspond to the birational model $\calM_S$, and the other chamber is the image of $\Nef(\calM_S[5,2,6])$ via the birational map given by crossing the wall $W_l$. See Figure \ref{FigNSMS}.
			\end{proposition}

			\begin{proof}
				First, note that $w=ae_1 + be_2$, for $a,b\in\matZ$, satisfies $w^2=0$ if and only if $16a^2=4b^2=0$. Hence the cone $\overline{\Pos}(\calM_S)$ is the cone generated, up to a sign, by $\{e_1+2e_2,e_1-2e_2\}$ and contains either $e_1$ or its inverse. Pick $\alpha=1,\beta=0$. It gives a well-defined generic stability condition $\sigma\coloneqq\sigma_{1,0}$, and its image $A\coloneqq l(\sigma)$ is an ample class (Theorem \ref{ThmBMChambStabisChambMov}). By computations, we find 
				$$A=\frac{1}{425}(16,-13,128)=  \frac{1}{425}(21e_1+8e_2).$$
				We deduce that $\overline{\Pos}(\calM_S)$ is generated by $\{e_1+e_2, e_1-2e_2\}$. Now the image of the vertical wall $W_v$ is given by $v^\perp \cap (2,1,4)^\perp = \matR e_1$. Hence $\overline{\Mov}(\calM_S)$ is the upper half-cone of $\overline{\Pos}(\calM_S)$ cut out by $\matR e_1$.
				
				Finally, the image of the circular wall $W_l$ is given by $v^\perp \cap (3,1,3)^\perp = \matR(16,3,0)$, and we obtain two chambers in $\overline{\Mov}(\calM_S)$. The chamber containing the ample class $A$ is $\Nef(\calM_S)$, and the other chamber corresponds to the unique other birational model $\calM_S[5,2,6]$. 
				
			\end{proof}

			Figure \ref{FigNSMS} represents the positive cone $\overline{\Pos}(\calM_S)$, and Figure \ref{FigStabS} represents the corresponding walls (in green and magenta) and chambers (in red and blue) in the $(\beta,\alpha)$-plane in $\Stab^+(S)$. The movable cone $\Mov(\calM_S)$ in Figure \ref{FigNSMS} is the upper half cone, composed of $\Nef(\calM_S)$ in red and (the flopped image of) $\Nef(\calM_S[5,2,6])$ in blue. The two lower chambers are their reflections in $\matR e_1$.

			\newpage

			\begin{figure}[tbp]
				\begin{center}
					
					\begin{tikzpicture}[scale=2.5]
						
						
						\draw[color=red!0,fill=red!30] (-0.5,0) rectangle (2,1);
						\draw[color=red!0,fill=red!15] (2,0) rectangle (4.5,1);

						\draw[->] (-0.5,0) -- (4.5,0);
						\draw[->] (2,0) -- (2,1);
						
						\draw[ultra thick, color=magenta, fill=blue!35] (0,0) arc (180:0:12/16);
						\draw[ultra thick, color=magenta!50, fill=blue!15] (4,0) arc (0:180:12/16);
						
						\draw[thick,->] (-0.5,0) -- (4.5,0);
						\draw[ultra thick, color=teal, ->] (2,0) -- (2,1);

						\node[anchor=south] at (2,1) {$\alpha$};
						\node[anchor=west] at (4.5,0) {$\beta$};
						
						\draw[fill, color=black] (0,0) circle (0.02);
						\node[anchor=north] at (0,0) {$0$};
						\draw[fill, color=black] (14/16,0) circle (0.02);
						\node[anchor=north] at (14/16,0) {$3/16$};
						\draw[fill, color=black] (52/16,0) circle (0.02);
						\node[anchor=north] at (52/16,0) {$13/16$};
						\draw[fill, color=black] (2,0) circle (0.02);
						\node[anchor=north] at (2,0) {$1/2$};
						\draw[fill, color=black] (4,0) circle (0.02);
						\node[anchor=north] at (4,0) {$1$};
					\end{tikzpicture}
					
					\caption{The $(\beta,\alpha)$-plane in $\Stab^+(S)$.}
					\label{FigStabS}
				\end{center}
			\end{figure}
			

			\begin{figure}[tbp]
				\begin{center}
					\begin{tikzpicture}[scale=2.5]
						
						
						\draw[color=blue!0, fill=blue!35] (0,0) -- (1,2) -- (5/4,2) -- cycle;
						\draw[color=blue!0, fill=blue!15] (0,0) -- (1,-2) -- (5/4,-2) -- cycle;
						
						\draw[color=red!0, fill=red!30] (0,0) -- (5/4,2) -- (2,0) -- cycle;
						\draw[color=red!0, fill=red!15] (0,0) -- (5/4,-2) -- (2,0) --cycle;
						
						\draw[->] (0,0) -- (0,2);
						\draw[ultra thick,color=teal,->] (0,0) -- (2,0);
						\draw (0,0) -- (0,-2);

						\draw[ultra thick, color=magenta,->] (0,0) -- (5/4,2);
						\draw[ultra thick, color=magenta!50,->] (0,0) -- (5/4,-2);
						\draw[ultra thick,->] (0,0) -- (1,2);
						\draw[ultra thick,->] (0,0) -- (1,-2);
						\node[anchor=west] at (2,0) {$\matR e_1$};
						\node[anchor=east] at (0,2) {$\matR e_2$};
						\node[anchor=east] at (0.95,2) {$\matR (2e_2+e_1)$};
						\node[anchor=west] at (5/4,2) {$\matR (8e_2+5e_1)$};
					\end{tikzpicture}
					
					\caption{The positive cone  $\overline{\Pos}(\calM_S)$.}
					\label{FigNSMS}
				\end{center}
			\end{figure}

			\bibliographystyle{abbrv}
			\bibliography{bibLagFib}

\begin{thebibliography}{10}

\bibitem{BarthPropStableRank2}
W.~{Barth}.
\newblock {Some properties of stable rank-2 vector bundles on
  \(\mathbb{P}_n\)}.
\newblock {\em {Math. Ann.}}, 226:125--150, 1977.

\bibitem{BMSpaceOfStabCondOnLocalProjectivePlane}
A.~{Bayer} and E.~{Macr\`{\i}}.
\newblock {The space of stability conditions on the local projective plane}.
\newblock {\em {Duke Math. J.}}, 160(2):263--322, 2011.

\bibitem{BMMMPwallcrossing}
A.~{Bayer} and E.~{Macr\`{\i}}.
\newblock {MMP for moduli of sheaves on \(K3\)s via wall-crossing: nef and
  movable cones, Lagrangian fibrations}.
\newblock {\em {Invent. Math.}}, 198(3):505--590, 2014.

\bibitem{BeauvilleSystemesHamiltoniens}
A.~Beauville.
\newblock Syst\`emes hamiltoniens compl\`etement int\'{e}grables associ\'{e}s
  aux surfaces {$K3$}.
\newblock In {\em Problems in the theory of surfaces and their classification
  ({C}ortona, 1988)}, Sympos. Math., XXXII, pages 25--31. Academic Press,
  London, 1991.

\bibitem{BeauvilleFanoThreefoldK3}
A.~Beauville.
\newblock Vector bundles on {F}ano threefolds and {K}3 surfaces.
\newblock {\em ArXiv e-print}, math.AG/1906.03594, 2019.

\bibitem{BrambillaFaenzig7}
M.~C. Brambilla and D.~Faenzi.
\newblock Vector bundles on {F}ano threefolds of genus 7 and {B}rill-{N}oether
  loci.
\newblock {\em Internat. J. Math.}, 25(3):1450023, 59, 2014.

\bibitem{BrambillaFaenzig9}
M.~C. Brambilla and D.~Faenzi.
\newblock Rank-two stable sheaves with odd determinant on {F}ano threefolds of
  genus nine.
\newblock {\em Math. Z.}, 275(1-2):185--210, arXiv version at
  \url{https://arxiv.org/abs/0905.1803}, 2013.

\bibitem{BridgelandStabCondTriangCat}
T.~{Bridgeland}.
\newblock {Stability conditions on triangulated categories}.
\newblock {\em {Ann. Math. (2)}}, 166(2):317--345, 2007.

\bibitem{BridgelandStabCondK3}
T.~{Bridgeland}.
\newblock {Stability conditions on \(K3\) surfaces}.
\newblock {\em {Duke Math. J.}}, 141(2):241--291, 2008.

\bibitem{CanonacoStellariTwistedFMFunctors}
A.~{Canonaco} and P.~{Stellari}.
\newblock {Twisted Fourier-Mukai functors}.
\newblock {\em {Adv. Math.}}, 212(2):484--503, 2007.

\bibitem{Caldararuthesis}
A.~H. C\u{a}ld\u{a}raru.
\newblock {\em Derived categories of twisted sheaves on {C}alabi-{Y}au
  manifolds}.
\newblock ProQuest LLC, Ann Arbor, MI, 2000.
\newblock Thesis (Ph.D.)--Cornell University.

\bibitem{DebarreHKManifolds}
O.~{Debarre}.
\newblock {Hyperk\"ahler manifolds}.
\newblock {\em arXiv preprint}, \url{https://arxiv.org/abs/1810.02087}, Oct
  2018.

\bibitem{FMOGSGeomAntisympInvolI}
L.~Flapan, E.~Macr{\`{\i}}, K.~G. O'Grady, and G.~Sacc{\`a}.
\newblock The geometry of antisymplectic involutions. {I}.
\newblock {\em Math. Z.}, 300(4):3457--3495, 2022.

\bibitem{HartshorneAlgebraicGeometry}
R.~Hartshorne.
\newblock {\em Algebraic geometry}.
\newblock Springer-Verlag, New York-Heidelberg, 1977.
\newblock Graduate Texts in Mathematics, No. 52.

\bibitem{HartshorneStableReflexiveSheaves}
R.~{Hartshorne}.
\newblock {Stable reflexive sheaves}.
\newblock {\em {Math. Ann.}}, 254:121--176, 1980.

\bibitem{HassetTschinkelMovingAmpleConesHKFourfolds}
B.~Hassett and Y.~Tschinkel.
\newblock Moving and ample cones of holomorphic symplectic fourfolds.
\newblock {\em Geom. Funct. Anal.}, 19(4):1065--1080, 2009.

\bibitem{bibhoppecriterion}
H.~J. Hoppe.
\newblock Generischer {S}paltungstyp und zweite {C}hernklasse stabiler
  {V}ektorraumb\"{u}ndel vom {R}ang {$4$} auf {${\bf P}_{4}$}.
\newblock {\em Math. Z.}, 187(3):345--360, 1984.

\bibitem{HuybrechtsFMTransform}
D.~Huybrechts.
\newblock {\em Fourier-{M}ukai transforms in algebraic geometry}.
\newblock Oxford Mathematical Monographs. The Clarendon Press, Oxford
  University Press, Oxford, 2006.

\bibitem{HuybrechtsLecturesK3}
D.~Huybrechts.
\newblock {\em Lectures on {K}3 surfaces}, volume 158 of {\em Cambridge Studies
  in Advanced Mathematics}.
\newblock Cambridge University Press, Cambridge, 2016.

\bibitem{HuybrechtsLehnModuliofsheaves}
D.~Huybrechts and M.~Lehn.
\newblock {\em The geometry of moduli spaces of sheaves}.
\newblock Cambridge Mathematical Library. Cambridge University Press,
  Cambridge, second edition, 2010.

\bibitem{HuybMatteiSBr}
D.~{Huybrehcts} and D.~{Mattei}.
\newblock {{T}he special {B}rauer group and twisted {P}icard varieties}.
\newblock {\em arXiv preprint}, \url{https://arxiv.org/abs/2310.04032}, Oct
  2023.

\bibitem{IlievRanestadGeometryOfLG36}
A.~{Iliev} and K.~{Ranestad}.
\newblock {Geometry of the Lagrangian Grassmannian LG(3,6) with applications to
  Brill-Noether loci}.
\newblock {\em {Mich. Math. J.}}, 53(2):383--417, 2005.

\bibitem{IskovskikhProkhorovFanoVarieties}
V.~A. {Iskovskikh} and Y.~G. {Prokhorov}.
\newblock {Fano varieties}.
\newblock In {\em {Algebraic geometry V: Fano varieties. Transl. from the
  Russian by Yu. G. Prokhorov and S. Tregub}}, pages 1--245. Berlin: Springer,
  1999.

\bibitem{KuznetsovHyperplaneSections}
A.~Kuznetsov.
\newblock Hyperplane sections and derived categories.
\newblock {\em Izv. Ross. Akad. Nauk Ser. Mat.}, 70(3):23--128, 2006.

\bibitem{KuznetsovHPD}
A.~{Kuznetsov}.
\newblock {Homological projective duality}.
\newblock {\em {Publ. Math., Inst. Hautes \'Etud. Sci.}}, 105:157--220, 2007.

\bibitem{LiAGInterpretationDonaldsonPolynInv}
J.~Li.
\newblock Algebraic geometric interpretation of {D}onaldson's polynomial
  invariants.
\newblock {\em J. Differential Geom.}, 37(2):417--466, 1993.

\bibitem{MaciociaComputationsWalls}
A.~{Maciocia}.
\newblock {Computing the walls associated to Bridgeland stability conditions on
  projective surfaces}.
\newblock {\em {Asian J. Math.}}, 18(2):263--280, 2014.

\bibitem{MacriSchmidtLecturesBridgelandStability}
E.~{Macr\`\i} and B.~{Schmidt}.
\newblock Lectures on {B}ridgeland stability.
\newblock In {\em Moduli of curves}, volume~21 of {\em Lect. Notes Unione Mat.
  Ital.}, pages 139--211. Springer, Cham, 2017.

\bibitem{stacks-project}
T.~{Stacks Project Authors}.
\newblock \textit{Stacks Project}.
\newblock \url{https://stacks.math.columbia.edu}, 2018.

\bibitem{YoshiokaSomeExamplesMukaiReflectionsK3Surf}
K.~Yoshioka.
\newblock Some examples of {M}ukai's reflections on {$K3$} surfaces.
\newblock {\em J. Reine Angew. Math.}, 515:97--123, 1999.

\bibitem{YoshiokaModuliSpacesSheavesAbelianSurfaces}
K.~{Yoshioka}.
\newblock {Moduli spaces of stable sheaves on abelian surfaces}.
\newblock {\em {Math. Ann.}}, 321(4):817--884, 2001.

\bibitem{YoshiokaModuliTwistedSheaves}
K.~Yoshioka.
\newblock Moduli spaces of twisted sheaves on a projective variety.
\newblock In {\em Moduli spaces and arithmetic geometry}, volume~45 of {\em
  Adv. Stud. Pure Math.}, pages 1--30. Math. Soc. Japan, Tokyo, 2006.

\bibitem{YoshiokaStabFMTransf}
K.~Yoshioka.
\newblock Stability and the {Fourier}-{Mukai} transform. {II}.
\newblock {\em Compos. Math.}, 145(1):112--142, 2009.

\end{thebibliography}
		\end{document}